\documentclass[final,leqno]{siamltex}

\usepackage{amsfonts}
\usepackage{amsmath}
\usepackage{bbm}
\usepackage{graphicx}
\usepackage{color}
\DeclareMathOperator*{\esssup}{ess\,sup}

\newtheorem{assumption}{Assumption}
\newtheorem{remark}{Remark}

\title{Systems of ergodic BSDEs arising in regime switching forward performance processes}

\author{Ying Hu\thanks{%
Univ Rennes, CNRS, IRMAR - UMR 6625, F-35000 Rennes, France, and School of Mathematical Sciences, Fudan
University, Shanghai 200433, China. Partially supported by Lebesgue
Center of Mathematics ``Investissements d'avenir"
program-ANR-11-LABX-0020-01, by ANR CAESARS (Grant No. 15-CE05-0024)
and by ANR MFG (Grant No. 16-CE40-0015-01). Email:
\texttt{ying.hu@univ-rennes1.fr} } \and
Gechun Liang\thanks{%
Department of Statistics, University of Warwick, Coventry, CV4 7AL,
U.K. Partially supported by Royal Society International Exchanges
(Grant No. 170137). Email: \texttt{g.liang@warwick.ac.uk} } \and
Shanjian Tang\thanks{%
Department of Finance and Control Sciences, School of Mathematical
Sciences, Fudan University, Shanghai 200433, China. Partially
supported by National Science Foundation of China (Grant No.
11631004) and National Key R\&D Program of China (Grant No. 2018YFA0703903). Email:
\texttt{sjtang@fudan.edu.cn}}}

\begin{document}

\maketitle

\begin{abstract}
We introduce and solve a new type of quadratic backward stochastic
differential equation systems defined in an infinite time horizon,
called \emph{ergodic BSDE systems}. Such systems arise naturally as
candidate solutions to characterize forward performance processes
and their associated optimal trading strategies in a regime
switching market. In addition, we develop a connection between the
solution of the ergodic BSDE system and the long-term growth rate of
classical utility maximization problems, and use the ergodic BSDE
system to study the large time behavior of PDE systems with
quadratic growth Hamiltonians.
\end{abstract}



\pagestyle{myheadings} \thispagestyle{plain} \markboth{Ying Hu,
Gechun Liang and Shanjian Tang}{Systems of ergodic BSDEs}


\section{Introduction}

This paper introduces a new class of quadratic backward stochastic
differential equation (BSDE for short) systems in an \emph{infinite
time horizon}, called \emph{ergodic BSDE systems}. The systems arise in  our solution of forward performance processes for
portfolio optimization problems in a regime switching market. We
show that ergodic BSDE systems are natural candidates for the
characterization of forward performance processes and associated
optimal strategies in a financial market with multiple regimes.

Let us first recall that an infinite horizon BSDE typically takes
the form
\begin{equation}\label{infhorizon_BSDE}
dY_t=-F(t,Y_t,Z_t)dt+(Z_t)^{tr}dW_t, \quad t\geq 0
\end{equation}
where $F$ is called the driver of the equation, and
$W$ is a $d$-dimensional Brownian motion as the driving noise of the
equation. In contrast to the  case of a finite time horizon $[0,T]$, the infinite horizon
BSDE (\ref{infhorizon_BSDE}) is defined over all time horizons and
may be ill posed, even if the driver $F$ is Lipschitz continuous in
both $Y$ and $Z$. It has been solved in \cite{Briand2}
under a strictly monotone condition on the driver, a
typical one of which  reads
$$F(t,Y_t,Z_t)=f(t,Z_t)-\rho Y_t,$$
for some constant $\rho>0$. Then, it has been shown in
\cite{Briand2} that (\ref{infhorizon_BSDE}) admits a unique bounded
solution $(Y,Z)$ adapted to the Brownian filtration, if $f$ is Lipschitz continuous in $Z$. If $f$ has a quadratic growth in $Z$, it has been further treated in~\cite{Briand0}.

 Note that only bounded solutions are concerned here, for unbounded solutions to BSDE~(\ref{infhorizon_BSDE}) are not unique in general. The restriction within bounded solutions is also useful in the study of the Markovian  BSDE~ (\ref{infhorizon_BSDE}) and its asymptotic property. Indeed, in a Markovian framework where $f(t,Z_t)=f(V_t,Z_t)$ with $V$ being the
underlying forward diffusion process, the solution $(Y,Z)$ admits a Markovian representation $(Y_t,Z_t)=(y(V_t),z(V_t))$ for some pair of  measurable functions $y(\cdot)$ and $z(\cdot)$, and has been shown in \cite{HU2} and
later in \cite{HU1} that, when $\rho\rightarrow 0$,
the bounded Markovian solution to BSDE (\ref{infhorizon_BSDE}) converges to the Markovian solution of the \emph{ergodic BSDE}
\begin{equation}\label{ergodic_BSDE}
dY_t=-(f(V_t,Z_t)-\lambda)dt+(Z_t)^{tr}dW_t,\quad t\geq 0.
\end{equation}
Here, the constant $\lambda$ constitutes one part of the
solution to (\ref{ergodic_BSDE}), and has a stochastic control
interpretation as the value of an ergodic control problem. The
ergodic BSDE (\ref{ergodic_BSDE}) has been widely used to study
the large time behavior of solutions of their finite horizon counterparts (see,
for example, \cite{Hu11} and \cite{Pham}).

Both ergodic BSDE~(\ref{ergodic_BSDE}) and infinite horizon
BSDE~(\ref{infhorizon_BSDE}) turn out to be natural
candidates for the characterization of forward performance processes
and their associated optimal portfolio strategies in portfolio
optimization problems. Forward performance processes were introduced
and developed in \cite{MZ0,MZ-Kurtz,MZ1,MZ2}. They complement the
classical expected utility paradigm in which the utility is a
deterministic function chosen at a single terminal time. The value
function process is, in turn, constructed backwards in time, as the
dynamic programming principle yields. As a result, there is limited
flexibility to incorporate updating of risk preferences, rolling
horizons, learning, and other realistic ``forward in nature"
features if one requires that time-consistency is being preserved at
all times. Forward performance processes alleviate some of these
shortcomings and offer the construction of a genuinely dynamic
mechanism for evaluating the performance of investment strategies as
the market evolves across (arbitrary) trading horizons. See also
\cite{Henderson2, Jan, Nadtochiy-Tehranchi, Shkolnikov-Sircar-Z, ZZ, Zit} for their
developments and various applications.

The construction of \emph{(Markovian) forward performance processes} is, however,
difficult, due to the ill-posed nature and degeneracy of the
corresponding (stochastic) partial differential equations (see
\cite{ElKaroui}). This difficulty has been recently overcome in
\cite{LZ}, which shows that Markovian forward performance processes in
homothetic form can be effectively constructed via the Markovian solutions of
the equations like (\ref{infhorizon_BSDE}) and (\ref{ergodic_BSDE}).
It bypasses a number of aforementioned difficulties inherited in the
associated SPDE. See also \cite{CHLZ} for a further development of
this method to study forward entropic risk measures.

Our aim herein is to generalize both (\ref{infhorizon_BSDE}) and
(\ref{ergodic_BSDE}) from scalar-valued to vector-valued equations,
i.e. systems of equations. The corresponding BSDE systems are
motivated by the construction of \emph{Markovian forward performance processes in a
regime switching market}. Due to the interactions of different market
regimes through a given Markov chain, the corresponding infinite
horizon BSDE system for a Markovian forward performance process is expected to
take the form
\begin{equation}\label{infhorizon_BSDE_system_0}
dY_t^i=-f^i(V_t,Z_t^i)dt-\sum_{k\in
I}q^{ik}(e^{Y_t^k-Y_t^{i}}-1)dt+(Z_t^i)^{tr}dW_t,\vspace{-0.2cm}
\end{equation}
for $t\geq 0$ and $i\in I:=\{1,2,\dots,m^0\}$, where $q^{ik}$ is the
transition rate from market regime $i$ to $k$. The second term on
the right hand side of (\ref{infhorizon_BSDE_system_0}) couples all
the equations together and represents the interaction of different
market regimes. A similar feature has also appeared in \cite{Bec0}
and \cite{Bec}, where the authors studied classical utility
maximization in a regime switching framework and derived a finite
horizon BSDE system.

However, different from the finite horizon case, the infinite
horizon BSDE system (\ref{infhorizon_BSDE_system_0}) is ill posed.
Indeed, in a single regime case, (\ref{infhorizon_BSDE_system_0})
then reduces to a scalar-valued BSDE, and the strictly monotone
condition fails to hold. To overcome this difficulty, we modify
(\ref{infhorizon_BSDE_system_0}) by adding a discount term $\rho
Y_t^i$ in the driver (see (\ref{infhorizon_BSDE_system}) in section
2), which serves the role of strict monotonicity. Although this
additional discount term makes the modified BSDE system well posed,
it however distorts the original problem. As a result, the solution
of the modified BSDE system will no longer correspond to a forward
performance process.

As a first contribution, we construct Markovian regime switching forward
performance processes in homothetic form via the asymptotic limit of
the infinite horizon BSDE system (\ref{infhorizon_BSDE_system}),
that is, the ergodic BSDE system (\ref{ergodic_BSDE_system}) (see
Theorem \ref{theorem_ergodic_representation}). Both BSDE systems
(\ref{infhorizon_BSDE_system}) and (\ref{ergodic_BSDE_system}) are
new\footnote{Recently, \cite{Cohen} also introduced an ergodic BSDE
system motivated from non-zero sum games. However, the structure of
their system is different from ours. In particular, there is no
comparison theorem for their system.}. They are introduced for the
first time for the characterization of regime switching forward
performance processes. In particular, we show that when there is a
single regime, our representation of forward performance processes
will recover the ergodic BSDE representation appearing in \cite{LZ}.

Our second contribution is about solvability of the infinite horizon
BSDE system (\ref{infhorizon_BSDE_system}). Since the driver $f^i$
has quadratic growth in $Z^i$, the standard Lipschitz estimates do
not apply to our system. Instead, we first apply a truncation
technique and derive \emph{a priori} estimates for the solutions,
and subsequently show that the truncation constants coincide with
the constants appearing in the \emph{a priori} estimates. For this,
we make an extensive use of the multidimensional comparison theorem
for BSDE systems, which was firstly developed in \cite{HuPeng}. An
essential idea herein is to use the bounded solution of an auxiliary
ODE (not system!) as a universal bound to control all the solution
components of the BSDE system.

We then derive the ergodic BSDE system (\ref{ergodic_BSDE_system})
as the asymptotic limit of the infinite horizon BSDE system
(\ref{infhorizon_BSDE_system}). This ergodic BSDE system, on one
hand, characterizes the regime switching forward performance
processes and, on the other hand, is also a natural extension of the
ergodic equation introduced in \cite{HU2}. Herein, a new feature is
that all the equation components have a common ergodic constant
$\lambda$ as a part of the solution. Similar to \cite{HU2}, we apply
the perturbation technique to construct a sequence of approximate
solutions to the ergodic BSDE system. However, the commonly used
Girsanov's transformation method does not imply the uniqueness of
the solution due to different probability measures induced by each
equation component. Instead, we prove the uniqueness of the solution
by first converting the ergodic BSDE system
(\ref{ergodic_BSDE_system}) to a scalar-valued ergodic BSDE driven
by the Brownian motion and an exogenously given Markov chain and
then using the Girsanov's transformation under the Brownian motion
and the Markov chain (see Appendix \ref{Appendix_B}).

Our third contribution is about a stochastic control representation
for the ergodic constant $\lambda$ (see Proposition
\ref{propositionLambda}). We show that it corresponds to the
long-term growth rate of a risk-sensitive optimization problem in a
regime switching framework. This, in turn, connects with the
long-term growth rate of a regime switching utility maximization
problem. Thus, our result also unveils an intrinsic connection
between forward performance processes and classical expected
utilities in a market with multiple regimes.

Our last contribution is using the ergodic BSDE system
(\ref{ergodic_BSDE_system}) to study the large time behavior of solutions to a
class of PDE systems with quadratic growth Hamiltonians (see Theorem
\ref{theorem_large_time_behavior}). Those PDE systems are often used
to characterize the utility indifference prices of financial
derivatives in a regime switching market (see \cite{Bec0} and
\cite{Bec}). We show that the solution of the PDE system will
converge to the solution of the ergodic BSDE system exponentially
fast. To the best of our knowledge,  this is the first convergence
rate result for the large time behavior of PDE systems.

Turning to literature about the quadratic BSDE (systems), most of
the existing results are only for a finite time horizon. The scalar
equation with bounded terminal data was first solved in
\cite{Kobylanski} and was applied to solve utility maximization
problems in \cite{him}. See also \cite{Briand1, Morlais, Tevzadze} for extensions. The case with unbounded terminal
data is more challenging and was solved in \cite{BH1, BH2, Delbaen0}, with \cite{dhr} and \cite{dhr2} further showing the
uniqueness of the solution. Their applications can be found in
\cite{Barrieu} and \cite{HLT}. Recently, there have been a renewed
interest in the corresponding quadratic BSDE systems due to their
various applications in equilibrium problems, price impact models
and non-zero sum games (see, for example, \cite{CN, HuTang, JKL, KP, KP2, XZ} with more references
therein). In spite of all the aforementioned results, our paper
seems to be the first to introduce and solve quadratic BSDE systems
in an infinite time horizon.\\

The paper is organized as follows. Section 2 introduces an infinite
horizon BSDE system with quadratic growth drivers. Section 3 studies
its asymptotic limit, which leads to an ergodic BSDE system. Section
4 applies the ergodic BSDE system to construct Markovian forward performance
processes in a regime switching market. Section 5 applies the
ergodic BSDE system to study the large time behavior of a PDE
system. Section 6 then concludes. For the reader's convenience, we
also provide a proof of the multidimensional comparison theorem in
the appendix.


\section{System of infinite horizon quadratic BSDE}\label{section:BSDE}

Let $W$ be a $d$-dimensional Brownian motion on a probability space
$(\Omega ,\mathcal{F},\mathbb{P})$. Denote by $\mathbb{F}=(
\mathcal{F}_{t})_{t\geq 0}$ the augmented filtration generated by
$W$. Throughout this paper, we denote by
$A^{tr}$  the transpose of matrix $A$. Consider the
infinite horizon BSDE system: for $t\geq 0$ and $i\in I:=\{1,2,\dots,m^0\}$,
\begin{equation}\label{infhorizon_BSDE_system}
dY_t^i=-f^i(V_t,Z_t^i)dt-\sum_{k\in
I}q^{ik}(e^{Y_t^k-Y_t^{i}}-1)dt+\rho
Y_t^idt+(Z_t^i)^{tr}dW_t. \vspace{-0.2cm}
\end{equation}
By a solution to
(\ref{infhorizon_BSDE_system}), we mean a pair of adapted processes
$(Y^i,Z^i)_{i\in I}$ satisfying (\ref{infhorizon_BSDE_system}) in an
arbitrary time horizon.

To solve (\ref{infhorizon_BSDE_system}), we impose the following
assumptions on $f^i$.

\begin{assumption}\label{assumption_11}
There exist three constants $C_v,C_z$ and $K_f$ such that, for $i,k\in I$
and $v,\bar{v}, z,\bar{z}\in\mathbb{R}^d$,

  (i) $|f^i(v,z)-f^i(\bar{v},z)|\leq C_v(1+|z|)|v-\bar{v}|$;

 (ii) $|f^i(v,z)-f^i(v,\bar{z})|\leq
C_z(1+|z|+|\bar{z}|)|z-\bar{z}|$;

 (iii) $|f^i(v,0)|\leq K_f$.
\end{assumption}

Assumption \ref{assumption_11}(ii) implies that $f^i(v,z)$
has a quadratic growth in $z$. Thus, we are facing
\emph{a system of quadratic BSDEs defined in an infinite time
horizon}. The system is coupled through the coefficients $q^{ik}$,
 $i,k\in I$, which satisfy

\begin{assumption}\label{assumption_12}
The square matrix $\mathcal{Q}:=\{q^{ik}\}_{i,k\in I}$ is a
transition rate matrix satisfying (i) $\sum_{k\in I}q^{ik}=0$; (ii)
$q^{ik}\geq 0$ for $i\neq k$. Let $q^{\max}$ be the maximal
transition rate, i.e. $q^{\max}=\max_{i,k} q^{ik}$.
\end{assumption}


The infinite horizon BSDE system (\ref{infhorizon_BSDE_system}) is
coupled with a forward diffusion process $V$ satisfying

\begin{assumption}\label{assumption_13}
The underlying $d$-dimensional forward diffusion process $V$ is given by the
solution of the  mean-reverting SDE
\begin{equation}
dV_{t}=\eta (V_{t})dt+\kappa dW_{t}\text{,} \label{factor}
\end{equation}
where the drift coefficients $\eta (\cdot)$ satisfy a dissipative
condition, namely, there exists a constant $C_{\eta }>C_v$ such
that, for $v,\bar{v}\in \mathbb{R}^{d}$,
\begin{equation*}
(\eta (v)-\eta (\bar{v}))^{tr}(v-\bar{v})\leq -C_{\eta }|v-\bar{v}|^{2}\text{%
.}
\end{equation*}
Moreover, the volatility matrix $\kappa \in \mathbb{R}^{d\times d}$
is positive definite and normalized to $|\kappa |=1$.
\end{assumption}

The main result of this section is the following existence and uniqueness of
the solution to (\ref{infhorizon_BSDE_system}).

\begin{theorem}\label{theorem_infhorizon_BSDE_system}
Let Assumptions \ref{assumption_11}, \ref{assumption_12}, and \ref{assumption_13} be
satisfied. Then, there exists a unique bounded solution
$(Y^i,Z^i)_{i\in I}$ to the infinite horizon BSDE system
(\ref{infhorizon_BSDE_system}) satisfying
\begin{equation}\label{estimate_yz}
|Y^i_t|\leq K_y:=\frac{K_f}{\rho}\ \ \text{and}\ \ |Z_t^i|\leq
K_z:=\frac{C_v}{C_{\eta}-C_v}.
\end{equation}
\end{theorem}
\begin{remark} As explained in the introduction,
we restricted our discussion within bounded solutions to BSDE~(\ref{infhorizon_BSDE_system}). This is for the sake of (i) the uniqueness of its adapted solutions; (ii) the discussion of its Markovian solutions; (iii) the discussion of the asymptotic behavior of solutions to (\ref{infhorizon_BSDE_system}) as the time horizon goes to infinity.
\end{remark}

The rest of this section is devoted to the proof of Theorem
\ref{theorem_infhorizon_BSDE_system}.

\subsection{Sketch of the proof}

To construct a solution of (\ref{infhorizon_BSDE_system}), we follow
a truncation procedure and a stability analysis. To this end, we
first define two truncating functions $p:\mathbb{R}\rightarrow\mathbb{R}$ and $q:\mathbb{R}^{d}\rightarrow \mathbb{%
R}^{d}$ by
\begin{equation}
p(y):=\max\{-K_y,\min\{y,K_y\}\}\ \ \text{and}\ \ q(z):=\frac{\min \left\{ |z|,K_z\right\} }{|z|}z\mathbf{1}%
_{\{z\neq 0\}}. \label{truncation}
\end{equation}%
We consider the truncated system of (\ref{infhorizon_BSDE_system}),
namely,
\begin{equation}\label{infhorizon_BSDE_system_truncation}
dY_t^i=-f^i(V_t,q(Z_t^i))dt-\sum_{k\in
I}q^{ik}(e^{p(Y_t^k)-p(Y_t^{i})}-1)dt+\rho
Y_t^idt+(Z_t^i)^{tr}dW_t,\vspace{-0.2cm}
\end{equation}
for $t\geq 0$ and $i\in I$.

Assumption \ref{assumption_11} $(i)$ and $(ii)$ imply that the
function $f^i(\cdot,q(\cdot))$ is Lipschitz continuous, i.e.
\begin{equation}
|{f}^i(v,q(z))-{f}^i(\bar{v},q(z)|\leq \frac{C_{\eta }C_{v}}{C_{\eta }-C_{v}}|v-%
\bar{v}|,  \label{driver3}
\end{equation}%
and%
\begin{equation}
|f^i(v,q(z))-{f}^i(v,q(\bar{z})|\leq C_{z}\frac{C_{\eta }+C_{v}}{C_{\eta }-C_{v}}%
|z-\bar{z}|.  \label{driver4}
\end{equation}%
It is also immediate to verify that $\sum_{k\in
I}q^{ik}(e^{p(y^k)-p(y^{i})}-1)-\rho y^i$ is continuous and has
bounded derivatives except at finite many points.
Thus, the driver of the truncated system
(\ref{infhorizon_BSDE_system_truncation}) is Lipschitz continuous.

If, moreover, we can show that
(\ref{infhorizon_BSDE_system_truncation}) admits a solution, say
$(Y^i,Z^i)_{i\in I}$, with $|Y_t^i|\leq K_y$ and $|{Z}_{t}^i|\leq
K_z$, then $p(Y_t^i)=Y_t^i$ and $q({Z}^i_{t})={Z}^i_{t}$, for $t\geq
0$ and $i\in I$. In turn, the pair of processes $({Y}^i,{Z}^i)_{i\in
I}$ also solve the original infinite horizon BSDE system
(\ref{infhorizon_BSDE_system}).

Next, we construct a solution to
(\ref{infhorizon_BSDE_system_truncation}) by an approximation
procedure. For $m\geq 1$ and $t\in[0,m]$, we consider the
\emph{finite horizon} BSDE system
\begin{align}\label{fhorizon_BSDE_system_truncation}
Y_t^i(m)=&\int_t^m\left[f^i(V_s,q(Z_s^i(m)))+\sum_{k\in
I}q^{ik}(e^{p(Y_s^k(m))-p(Y_s^{i}(m))}-1)-\rho Y_s^i(m)\right]ds\notag\\
&-\int_t^m(Z_s^i(m))^{tr}dW_s.
\end{align}
For $t> m$, we define $Y_t^{i}(m)=Z_t^{i}(m)\equiv 0$. Note that
(\ref{fhorizon_BSDE_system_truncation}) is a standard BSDE system
with Lipschitz continuous driver, so it admits a unique solution
$(Y^{i}(m),Z^{i}(m))_{i\in I}$.

We shall first establish uniform bounds (independent of $m$) on
$Y^i(m)$ and $Z^i(m)$ in Section \ref{subsection_estimate}.
Subsequently, we shall show in Section
\ref{subsection_proof_of_theorem} that the pair of processes
$(Y^i(m),Z^i(m))_{m\geq 1}$ is a Cauchy sequence in an appropriate
space, whose limit then provides a solution to the infinite horizon
BSDE system (\ref{infhorizon_BSDE_system}). Moreover, the uniqueness
of the solution relies on the multidimensional comparison theorem
introduced in the next subsection.

\subsection{Multidimensional comparison theorem}

The multidimensional comparison theorem
for systems of BSDE, was first established in \cite{HuPeng}. A different proof is given in Appendix \ref{Appendix_A} for
the reader's convenience.

\begin{lemma}\label{comparsion_lemma} For $T>0$, consider a system of
BSDEs $(\xi^i,F^i,G^i)$ with the terminal data $\xi^i$ and the driver
$(F^i,G^i)$, namely,
$$Y^{i}_t=\xi^i+\int_t^T\left[F^i_s(Z_s^i)+G^i_s(Y_s^i,Y_s^{-i})\right]ds-\int_t^T(Z_s^i)^{tr}dW_s,\quad t\in [0,T], $$
where
$Y_s^{-i}:=(Y_s^{1},\dots,Y_s^{i-1},Y_s^{i+1},\dots,Y_s^{m^0})$. Let
$(\bar{Y}^i,\bar{Z}^i)$ be the solution of the system of
BSDEs $(\bar{\xi}^i,\bar{F}^i,\bar{G}^i)$ with the
terminal data $\bar{\xi}^i$ and the driver $(\bar{F}^i,\bar{G}^i)$.
Suppose that

(i) both $\xi^i$ and $\bar{\xi}^i$ are square integrable and
satisfying $\xi^i\leq \bar{\xi}^i$ for $i\in I$;

(ii) there exist constants $C_f$ and $C_g$ such that, for $i\in
I$ and $z,\bar{z}\in\mathbb{R}^d$, $y=(y^i,y^{-i}),
\bar{y}=(\bar{y}^i,\bar{y}^{-i})\in\mathbb{R}^{m^0}$,
\begin{align}
|F^i_s(z)-F^i_s(\bar{z})|&\leq C_f|z-\bar{z}|\label{Lip_F},\\
|G^i_s(y^i,y^{-i})-G^i_s(\bar{y}^i,\bar{y}^{-i})|&\leq
C_g|y-\bar{y}|;\label{Lip_G}
\end{align}

(iii) the driver $G_s^{i}(y^i,{y}^{-i})$ is nondecreasing in all
of its components other than $y^i$, i.e. it is nondecreasing in
$y^{k}$, for $k\neq i$;

(iv) the following inequalities hold,
\begin{align}
F_s^{i}(\bar{Z}_s^i)&\leq \bar{F}_s^{i}(\bar{Z}_s^i),\label{compare_F}\\
G^i_s(\bar{Y}_s^i,\bar{Y}_s^{-i})&\leq
\bar{G}^i_s(\bar{Y}_s^i,\bar{Y}_s^{-i})\label{compare_G}.
\end{align}
Then, $Y_t^{i}\leq \bar{Y}_t^i$ for $t\in[0,T]$ and $i\in I$.
\end{lemma}

\begin{remark}
{Lemma~\ref{comparsion_lemma}, and its proof, simply correct a minor loss of inefficiency in the arguments developed in~\cite{HuPeng}. }  In \cite{HuPeng},  the Lipschitz conditions (\ref{Lip_F}) and (\ref{Lip_G}) are required to hold also for  $(\bar{F}_i,\bar{G}_i)$,  and  both inequalities (\ref{compare_F}) and (\ref{compare_G}) are required to hold for all $z^i\in\mathbb{R}$ and $y=(y^i,y^{-i})\in\mathbb{R}^{m^0}$.  In Lemma~\ref{comparsion_lemma},  the Lipschitz conditions on $(\bar{F}_i,\bar{G}_i)$ are not necessary,  and both inequalities (\ref{compare_F}) and (\ref{compare_G}) are required to hold only at the solution $(\bar{Y}^i,\bar{Z}^i)$.  Such an improvement is crucial and tailor made for our later use.
\end{remark}

\subsection{{A priori} estimates}\label{subsection_estimate}

We show that the pair of processes $(Y^{i}(m),Z^{i}(m))_{i\in I}$,
as the solution to the finite horizon BSDE system
(\ref{fhorizon_BSDE_system_truncation}), have the estimates
\begin{equation}\label{aprior_estimate} |Y^i_t(m)|\leq K_y\ \
\text{and}\ \ |Z^i_t(m)|\leq K_z,
\end{equation}
where the constants $K_y$ and $K_z$, independent of $m$, are given
in Theorem \ref{theorem_infhorizon_BSDE_system}.

\emph{The boundedness of $Y^i(m)$.} For $z\in\mathbb{R}^{d}$ and
$y=(y^{i},y^{-i})\in\mathbb{R}^{m^0}$, let
$$F_s^{i}(z):=f^i(V_s,q(z))\ \ \text{and}\ \ G_s^i(y^i,{y}^{-i}):=\sum_{k\in
I}q^{ik}(e^{p(y^k)-p(y^i)}-1)-\rho y^i.$$

Note that both $F^{i}_s(z)$ and $G^i_s(y^i,{y}^{-i})$ are Lipschitz
continuous, and $G^{i}_s(y^i,y^{-i})$ is nondecreasing in $y^k$ for
$k\neq i$. Moreover, by Assumption \ref{assumption_11}(iii),
$F_s^i(0)\leq K_f$ and $G^{i}_s(\bar{Y}_s,\bar{Y}_s^{-i})=-\rho
\bar{Y}_s,$ where
$\bar{Y}^{-i}:=(\underbrace{\bar{Y},\dots,\bar{Y}}_{m^0-1})$ and
$\bar{Y}$ solves the ODE\vspace{-0.2cm}
$$\bar{Y}_t=\int_t^m(K_f-\rho\bar{Y}_s)ds.$$
Consequently, it follows from Lemma \ref{comparsion_lemma} that
$Y_t^{i}(m)\leq \bar{Y}_t\leq \frac{K_f}{\rho}$, for $t\in[0,m]$ and
$i\in I$. Likewise, we also obtain that $Y_t^{i}(m)\geq
-\frac{K_f}{\rho}$, so $|Y_t^{i}(m)|\leq \frac{K_f}{\rho}=K_y$.
Hence, we have $p(Y_t^{i}(m))\equiv Y_t^i(m)$, i.e. the truncation
function $p(\cdot)$ does not play a role in BSDE system
(\ref{fhorizon_BSDE_system_truncation}).

\emph{The boundedness of $Z^i(m)$.} Denote by $V^{r,v}$ the solution of SDE~\eqref{factor} starting from $v\in\mathbb{R}^d$ at the initial time $r$, and by $(Y^{i,r,v}_t(m), Z^{i,r,v}_t(m)), t\in [r,T]$ the solution of BSDE~\eqref{fhorizon_BSDE_system_truncation} where the process $V$ is replaced with $V^{r,v}$.  Identically just as before, we have $|Y_t^{i,r,v}(m)|\le K_y.$

For $t\in[r, m]$ and $v,\bar{v}\in\mathbb{R}^d$, let
\begin{align*}
\delta Y_t^{i,r}(m):=Y_t^{i,r,v}(m)-Y_t^{i,r,\bar{v}}(m)\ \ \text{and}\ \
\delta Z_t^{i,r}(m):=Z_t^{i,r,v}(m)-Z_t^{i,r,\bar{v}}(m).
\end{align*}
It then follows from (\ref{fhorizon_BSDE_system_truncation}) that
\begin{align}\label{fhorizon_BSDE_system_truncation_2}
\delta Y_t^{i,r}(m)=&\int_t^m\left[f^i(V_s^{r,v},q(Z_s^{i,r,v}(m)))-f^i(V_s^{r, \bar{v}},q(Z_s^{i,r, \bar{v}}(m)))\right]ds\notag\\
&+\int_t^{m}\sum_{k\in
I}\left(q^{ik}(e^{Y_s^{k,r, v}(m)-Y_s^{i,r, v}(m)}-1)-q^{ik}(e^{Y_s^{k,r, \bar{v}}(m)-Y_s^{i,r, \bar{v}}(m)}-1)\right)ds\notag\\
&-\int_t^m\rho\, \delta Y_s^{i,r}(m)ds-\int_t^m(\delta
Z_s^{i,r}(m))^{tr}dW_s\notag\\
=&\int_t^{m}\left[F^{i,r}_{s}(\delta Z_s^{i,r}(m))+G_s^{i,r}(\delta
Y_s^{i,r}(m),\delta Y_s^{-i,r}(m))\right]ds-\int_t^m(\delta
Z_s^{i,r}(m))^{tr}dW_s,
\end{align}
where
\begin{align*}
F_s^{i,r}(z)=&\ f^i(V_s^{r,v},q(Z_s^{i,r,v}(m)))-f^i(V_s^{r,\bar{v}},q(Z_s^{i,r,v}(m)))\\
&\
+f^i(V_s^{r,\bar{v}},q(z+Z_s^{i,r, \bar{v}}(m)))-f^i(V_s^{r, \bar{v}},q(Z_s^{i,r, \bar{v}}(m)))
\end{align*}
and
$$G_s^{i,r}(y^{i},y^{-i})=\sum_{k\in I}q^{ik}\left(e^{y^{k}-y^{i}
+Y_s^{k,r, \bar{v}}(m)-Y_s^{i,r, \bar{v}}(m)}-e^{Y_s^{k,r, \bar{v}}(m)-Y_s^{i,r, \bar{v}}(m)}\right)-\rho
y^i,$$ for $z\in\mathbb{R}^d$ and
$y=(y^i,y^{-i})\in\mathbb{R}^{m^0}$, with $|y^i|\leq 2K_y$ for $i\in
I$.

Note that $F^{i,r}_s(z)$ and $G^{i,r}_s(y^i,{y}^{-i})$ are Lipschitz
continuous. Moreover, $G_s^{i,r}(0,0^{-i})=0$ and, by Assumption
\ref{assumption_11}(i) and the Lipschitz estimate (\ref{driver3}),
\begin{align*}
|F_s^{i,r}(0)|=&\ |f^i(V_s^{r,v},q(Z_s^{i,r,v}(m)))-f^i(V_s^{r,\bar{v}},q(Z_s^{i,r, v}(m)))|\\
\leq&\ \frac{C_vC_{\eta}}{C_{\eta}-C_v}|V^{r, v}_s-V^{r, \bar{v}}_s| \leq
\frac{C_vC_{\eta}}{C_{\eta}-C_v}e^{-C_{\eta}(s-r)}|v-\bar{v}|,\quad s\in [r,T],
\end{align*}
where the last inequality follows from the dissipative condition in
Assumption \ref{assumption_13} and Gronwall's inequality. Thus,
$(\delta Y^{i,r}(m),\delta Z^{i,r}(m))_{i\in I}$ is the unique solution to
(\ref{fhorizon_BSDE_system_truncation_2}). Furthermore, note that
$G^{i,r}_s(y^i,y^{-i})$ is nondecreasing in $y^k$ for $k\neq i$ and
$G^{i,r}_s(\bar{Y}_s^r,(\bar{Y}_s^r)^{-i})=-\rho \bar{Y}_s^r,$ where $\bar{Y}^r$
is the unique solution of  the ODE
$$Y_t=\int_t^m\left(\frac{C_vC_{\eta}}{C_{\eta}-C_v}e^{-C_{\eta}(s-r)}|v-\bar{v}|-\rho Y_s\right)ds, \quad t\in [r,T].$$
Consequently, from Lemma \ref{comparsion_lemma}, we have
\begin{align}\label{Lip_upper}
\delta Y_t^{i,r}(m)\leq
\bar{Y}_t^r&=\frac{C_vC_{\eta}}{C_{\eta}-C_v}\frac{e^{\rho t}(e^{-(\rho t+C_{\eta} (t-r))}-e^{-(\rho m+C_{\eta}(m-r))})}{\rho+C_{\eta}}|v-\bar{v}|\notag\\
&\leq\frac{C_v}{C_{\eta}-C_v}|v-\bar{v}|
\end{align}
for $t\in[r,m]$ and $i\in I$. Likewise, we also have
\begin{equation}\label{Lip_lower}
\delta Y_t^{i,r}(m)\geq\frac{-C_v}{C_{\eta}-C_v}|v-\bar{v}|.
\end{equation}

Note that the process $Y^{i,r, v}(m)$ admits a Markovian representation, i.e. there exists a measurable
function $\mathbf{y}^{i}(\cdot,\cdot;m)$ such that
$Y_t^{i, r, v}(m)=\mathbf{y}^{i}(t,V_t^{r,v};m)$ (see Theorem 4.1 in \cite{ElKaroui1997}). If the coefficient $\eta$ and the driver $f^i$ are further  continuously differentiable functions with bounded derivatives,  the function $v\mapsto \mathbf{y}^i(t,v;m)$ is continuously differentiable such that (see \cite[Corollary 4.1]{ElKaroui1997})
\begin{equation}\label{relation}
\kappa^{tr}\nabla_v
\mathbf{y}^i(t,V_t^{r,v};m)=Z_t^{i,r, v}(m).
\end{equation}  From  (\ref{Lip_upper}) and (\ref{Lip_lower}), we have for $(i, v_1,v_2)\in I \times \mathbb{R}^d\times \mathbb{R}^d,$
\begin{equation}\label{gradient_estimate}
|\mathbf{y}^{i}(t,v_1;m)-\mathbf{y}^{i}(t,v_2;m)|\leq K_z|v_1-v_2| \text{ with }  K_z:=\frac{C_v}{C_{\eta}-C_v}.
\end{equation}
In view of  Assumption
\ref{assumption_13} on $\kappa$, we have
$|Z_t^{i,r,v}(m)|\leq K_z$, which can be shown (via  mollifying the coefficient $\eta$ and the driver $f^i$ in a straightforward manner)  to hold also for our general $(\eta, f)$.
 Therefore, the \emph{a priori} estimates (\ref{aprior_estimate}) on
$Y^i(m)=Y^{i,0,v}(m)$ and $Z^i(m)=Z^{i,0, v}(m)$ have been proved.

\subsection{Proof of Theorem \ref{theorem_infhorizon_BSDE_system}}
\label{subsection_proof_of_theorem}

\emph{Existence.} We first prove that $(Y^{i}(m))_{m\geq 1}$ is a
Cauchy sequence. For $m\geq n\geq 1$ and $t\in[0,m]$, let
\begin{align*}
\delta Y_t^i(m,n):=Y_t^{i}(m)-Y_t^{i}(n)\ \ \text{and}\ \ \delta
Z_t^i(m,n):=Z_t^{i}(m)-Z_t^{i}(n).
\end{align*}
Since we have already shown in the last section that
$|Y_t^{i}(m)|\leq K_y$ and $|Z^{i}_t(m)|\leq K_z$, the truncation
functions $p(\cdot)$ and $q(\cdot)$ actually do not play any role in
(\ref{fhorizon_BSDE_system_truncation}), and we have
$(p(Y^i_t(m)),q(Z^i_t(m))=(Y_t^{i}(m),Z_t^i(m))$. In turn,
\begin{align}\label{fhorizon_BSDE_system_truncation_3}
\delta Y_t^{i}(m,n)
=&\int_t^m\left[f^i(V_s,Z_s^{i}(m))-f^i(V_s,Z_s^{i}(n))\right]ds
+\int_t^mf^i(V_s,0)\chi_{\{s\geq n\}}ds\notag\\
&+\int_t^m\sum_{k\in
I}\left(q^{ik}(e^{Y_s^{k}(m)-Y_s^{i}(m)}-1)-q^{ik}(
e^{Y_s^{k}(n)-Y_s^{i}(n)}-1)\right)ds\notag\\
&-\int_t^m\rho\delta Y_s^{i}(m,n)ds-\int_t^m(\delta
Z_s^{i}(m,n))^{tr}dW_s\notag\\
=&\int_t^{m}\left[F^i_{s}(\delta Z_s^i(m,n))+G_s^{i}(\delta
Y_s^{i}(m,n),\delta Y_s^{-i}(m,n))\right]ds\notag\\
&-\int_t^m(\delta Z_s^{i}(m,n))^{tr}dW_s,
\end{align}
where
\begin{equation*}
F_s^{i}(z)=f^i(V_s,z+Z_s^{i}(n))-f^i(V_s,Z_s^{i}(n))+f^i(V_s,0)\chi_{\{s\geq
n\}},
\end{equation*}
and
$$G_s^{i}(y^{i},y^{-i})=\sum_{k\in I}q^{ik}\left(e^{y^{k}-y^{i}
+Y_s^{k}(n)-Y_s^{i}(n)}-e^{Y_s^{k}(n)-Y_s^{i}(n)}\right)-\rho y^i,$$
for $z\in\mathbb{R}^d$ and $y=(y^{i},y^{-i})\in\mathbb{R}^{m^0}$,
with $|z|\leq 2K_z$ and $|y^{i}|\leq 2K_y$ for $i\in I$.

Following along similar arguments as in section
\ref{subsection_estimate}, we deduce that
(\ref{fhorizon_BSDE_system_truncation_3}) is with Lipschitz
continuous driver and, therefore, $(\delta Y^i(m,n),\delta
Z^i(m,n))_{i\in I}$ is the unique solution to
(\ref{fhorizon_BSDE_system_truncation_3}). Moreover, by Assumption
\ref{assumption_11}(iii), we have
$F_s^{i}(0)=f^i(V_s,0)\chi_{\{s\geq n\}}\leq K_f\chi_{\{s\geq n\}}$
and $G_s^{i}(\bar{Y}_s,\bar{Y}_s^{-i})=-\rho \bar{Y}_s$, with
$\bar{Y}$ solving the ODE
$$\bar{Y}_t=\int_t^m\left(K_f\chi_{\{s\geq n\}}-\rho\bar{Y}_s\right)ds$$
Hence, using Lemma \ref{comparsion_lemma}, we obtain
\begin{equation}\label{Cauchy_upper}
\delta Y_t^{i}(m,n)\leq \bar{Y}_t\leq K_f\int_{ n}^{m}e^{-\rho(s-t)}ds=\frac{K_f}{\rho}e^{\rho t}(e^{-\rho n}-e^{-\rho m}),
\end{equation}
for $t\in[0,m]$ and $i\in I$. Likewise, we also have
\begin{equation}\label{Cauchy_lower}
\delta Y_t^{i}(m,n)\geq -\frac{K_f}{\rho}e^{\rho t}(e^{-\rho n}-e^{-\rho m}).
\end{equation}
Sending $m,n\rightarrow \infty$, we obtain that, for any $T>0$,
$\sup_{t\in[0,T]}|\delta Y_t^{i}(m,n)|\rightarrow 0$ and, therefore,
there exists a limit process $Y^i$ such that $Y_t^{i}(m)\rightarrow
Y_t^i$ for almost every $(t,\omega)\in[0,\infty)\times\Omega$, with
$|Y_t^i|\leq K_y$.

To prove that $Z^i(m)$ is also a Cauchy sequence, we introduce the
Banach space
\[
\mathcal{L}^{2,\rho}:=\left\{(Z_t)_{t\geq 0}: Z\ \text{is
progressively measurable and}\ \mathbb{E}[\int_0^{\infty}e^{-2\rho
s}|Z_s|^2ds]<\infty\right\}.
\]

Applying It\^o's formula to $e^{-2\rho t}|\delta Y_t^{i}(m,n)|^2$
and using (\ref{fhorizon_BSDE_system_truncation_3}), we get
\begin{align}\label{fhorizon_BSDE_system_truncation_5}
&|\delta Y_0^{i}(m,n)|^2+\int_0^{m}e^{-2\rho s}|\delta Z_s^i(m,n)|^2ds\notag\\
=&\int_0^m\underbrace{2e^{-2\rho s}\delta Y_s^i(m,n)\left[f^i(V_s,Z_s^{i}(m))-f^i(V_s,Z_s^{i}(n))\right]}_{(I)}ds\notag\\
&+\int_0^m2e^{-2\rho s}\delta Y_s^i(m,n)f^i(V_s,0)\chi_{\{s\geq n\}}ds\notag\\
&+\int_0^m2e^{-2\rho s}\delta Y_s^i(m,n)\sum_{k\in
I}q^{ik}\left(e^{Y_s^{k}(m)-Y_s^{i}(m)}-
e^{Y_s^{k}(n)-Y_s^{i}(n)}\right)ds\notag\\
&-\int_0^m2e^{-2\rho s}\delta Y_s^i(m,n)(\delta
Z_s^{i}(m,n))^{tr}dW_s.
\end{align}
Furthermore, we apply the elementary inequality $2ab\leq
\frac{1}{\epsilon}|a|^2+\epsilon|b|^2$ to term $(I)$ and obtain
\begin{align*}
(I)\leq&\ \frac12e^{-2\rho
s}\frac{|f^i(V_s,Z_s^i(m))-f^i(V_s,Z_s^i(n))|^2}{C_z^2(1+2K_z)^2}
+2C_z^2(1+2K_z)^2e^{-2\rho s}|\delta Y_s^i(m,n)|^2\\
\leq&\ \frac12e^{-2\rho s}|\delta
Z_s^i(m,n)|^2+2C_z^2(1+2K_z)^2e^{-2\rho s}|\delta Y_s^i(m,n)|^2,
\end{align*}
where we also used Assumption \ref{assumption_11}(ii) and the
\emph{a priori} estimate (\ref{aprior_estimate}) on $Z^i(m)$ in the
second equality.

In turn, taking expectation on both sides of
(\ref{fhorizon_BSDE_system_truncation_5}) and using the \emph{a
priori} estimate (\ref{aprior_estimate}) on $Y^i(m)$ yield
\begin{align*}
&\frac{1}{2}\mathbb{E}\left[\int_0^{m}e^{-2\rho s}|\delta
Z_s^i(m,n)|^2ds\right]\\
\leq&\ 2C_z^2(1+2K_z)^2\mathbb{E}\left[\int_0^me^{-2\rho s}|\delta
Y_s^i(m,n)|^2ds\right]+ 2K_f\mathbb{E}\left[\int_n^me^{-2\rho
s}\delta Y_s^i(m,n)ds\right]\\
&+4m^0q^{\max}e^{2K_y}\mathbb{E}\left[\int_0^me^{-2\rho s}\delta
Y_s^i(m,n)ds\right].
\end{align*}
The dominated convergence theorem then implies $\delta
Z^i(m,n)\rightarrow 0$ in $\mathcal{L}^{2,\rho}$ and, therefore,
there exists a limit process $Z^i$ such that $Z^i(m)\rightarrow Z^i$
in $\mathcal{L}^{2,\rho}$, with $|Z^i_t|\leq K_z$.

It is standard to check that the pair of limit processes
$(Y^i,Z^i)_{i\in I}$ indeed satisfy the infinite horizon BSDE system
(\ref{infhorizon_BSDE_system}). See, for example, section 5 of
\cite{Briand2}.

\emph{Uniqueness.} Since both $Y^{i}$ and $Z^i$ are bounded, the
uniqueness of the bounded solution $(Y^i,Z^i)_{i\in I}$ to
(\ref{infhorizon_BSDE_system}) follows from the multidimensional
comparison theorem in Lemma \ref{comparsion_lemma}. Indeed, suppose
$(Y^i,Z^i)_{i\in I}$ and $(\bar{Y}^i,\bar{Z}^i)_{i\in I}$ are two
bounded solutions to (\ref{infhorizon_BSDE_system}). For $t\geq 0$,
let
\begin{align*}
\delta Y_t^i:=e^{-\rho t}(Y_t^{i}-\bar{Y}_t^{i})\ \ \text{and}\ \
\delta Z_t^i:=e^{-\rho t}(Z_t^{i}-\bar{Z}_t^{i}).
\end{align*}
For $T \geq t$, let $\varepsilon_T:=2K_ye^{-\rho T}$. Then, for
$0\leq t\leq T$,
\begin{align}\label{infhorizon_BSDE_system_111}
\delta Y_t^{i} =&\ \delta Y_T^i+\int_t^T e^{-\rho
s}\left[f^i(V_s,Z_s^{i})-f^i(V_s,\bar{Z}_s^{i})\right]ds
\notag\\
&+\int_t^Te^{-\rho s}\sum_{k\in
I}\left(q^{ik}(e^{Y_s^{k}-Y_s^{i}}-1)-q^{ik}(
e^{\bar{Y}_s^{k}-\bar{Y}_s^{i}}-1)\right)ds\notag\\
&-\int_t^T(\delta
Z_s^{i})^{tr}dW_s\notag\\
=&\ \delta Y_T^i+\int_t^{T}\left[F^i_{s}(\delta
Z_s^i)+G_s^{i}(\delta Y_s^{i},\delta
Y_s^{-i})\right]ds-\int_t^T(\delta Z_s^{i})^{tr}dW_s,
\end{align}
where
\begin{equation*}
F_s^{i}(z)=e^{-\rho s}[f^i(V_s,e^{\rho
s}z+\bar{Z}_s^{i})-f^i(V_s,\bar{Z}_s^{i})],
\end{equation*}
and
$$G_s^{i}(y^{i},y^{-i})=e^{-\rho s}\sum_{k\in I}q^{ik}\left(e^{e^{\rho
s}(y^{k}-y^{i})
+\bar{Y}_s^{k}-\bar{Y}_s^{i}}-e^{\bar{Y}_s^{k}-\bar{Y}_s^{i}}\right),$$
for $z\in\mathbb{R}^d$ and $y=(y^{i},y^{-i})\in\mathbb{R}^{m^0}$,
with $|z|\leq 2K_z$ and $|y^{i}|\leq 2K_y$ for $i\in I$.

We apply similar arguments as in section \ref{subsection_estimate}
to deduce that (\ref{infhorizon_BSDE_system_111}) is with Lipschitz
continuous driver and, therefore, $(\delta Y^i,\delta Z^i)_{i\in I}$
is the unique solution to (\ref{infhorizon_BSDE_system_111}).
Moreover, note that $$|\delta Y_T^i|\leq 2K_ye^{-\rho
T}=\varepsilon_T,\ \ F_s^i(0)=0\ \ \text{and}\ \
G_s^i(\varepsilon_T,\varepsilon_T^{-i})=0.$$ By Lemma
\ref{comparsion_lemma}, we deduce that $|\delta Y_t^i|\leq
\varepsilon_T$ and, therefore, $\delta Y_t^i=0$ by sending
$T\rightarrow\infty$. Consequently, $\delta Z_t^i=0$, which proves
the uniqueness of the solution to the infinite horizon BSDE system
(\ref{infhorizon_BSDE_system}).

\section{System of ergodic quadratic BSDEs}

We study the asymptotics of the infinite horizon BSDE system
(\ref{infhorizon_BSDE_system}) when $\rho\rightarrow 0$, which leads
to a new type of ergodic BSDE systems. The ergodic BSDE system will in turn be used to construct regime switching forward performance processes (section \ref{section:regime_switching_forward}) and obtain the large time behavior of PDE systems (section \ref{section: large_time}).
To this end, we require that
the transition rate matrix $\mathcal{Q}$ in Assumption
\ref{assumption_12} satisfies some sort of irreducible property.

\begin{assumption}\label{assumption_15} The transition rate matrix $\mathcal{Q}$
satisfies $q^{ik}>0$, for $i\neq k$. Let $q^{\min}>0$ be the minimal
transition rate, i.e. $q^{\min}=\min_{i\neq k}q^{ik}$.
\end{assumption}

We first show that, under Assumption \ref{assumption_15}, the
difference of any two components, say $Y^i$ and $Y^{j}$, of the
solution to (\ref{infhorizon_BSDE_system}) is actually bounded
uniformly in $\rho$ .

\begin{lemma}\label{lemma_difference}
Suppose that Assumptions \ref{assumption_11}-\ref{assumption_15} are
satisfied. For $i,j\in I$ and $t\geq 0$, let $\Delta
Y_t^{ij}=Y^i_t-Y^j_t$. Then,
\begin{equation}\label{diff_estimate_0}
|\Delta Y_t^{ij}|\leq
\frac{1}{q^{\min}}\left(K_f+\frac{C_vC_{\eta}C_z}{(C_{\eta}-C_v)^2}\right),
\end{equation}
with the constants $K_f, C_{v}, C_z$ as in Assumption
\ref{assumption_11}, and $C_{\eta}$ as in Assumption
\ref{assumption_13}.
\end{lemma}

\begin{proof} It suffices to prove that, for $m\geq 1$,
\begin{equation}\label{diff_estimate}
|\Delta Y_t^{ij}(m)|:=|Y_t^{i}(m)-Y_t^{j}(m)|\leq
\frac{1}{q^{\min}}\left(K_f+\frac{C_vC_{\eta}C_z}{(C_{\eta}-C_v)^2}\right).
\end{equation}
Then, (\ref{diff_estimate_0}) follows by sending
$m\rightarrow\infty$.

To this end, let $\Delta Z_t^{ij}(m)=Z^i_t(m)-Z^j_t(m)$. It is
immediate to check that the pair of processes $(\Delta
Y^{ij}(m),\Delta Z^{ij}(m))_{i,j\in I}$ satisfy
\begin{align}\label{fhorizon_BSDE_system_truncation_1}
\Delta Y_t^{ij}(m)=&\int_t^m\left[f^i(V_s,Z_s^i(m))-f^j(V_s,Z_s^j(m))\right]ds\notag\\
&+\int_t^{m}\sum_{k\in
I}\left(q^{ik}(e^{Y_s^k(m)-Y_s^{i}(m)}-1)-q^{jk}(e^{Y_s^k(m)-Y_s^{j}(m)}-1)\right)ds\notag\\
&-\int_t^m\rho\Delta Y_s^{ij}(m)ds-\int_t^m(\Delta
Z_s^{ij})^{tr}dW_s\notag\\
=&\int_t^{m}\left[F^{ij}_{s}(\Delta Z_s^{ij}(m))+G_s^{ij}(\Delta
Y_s^{ij}(m),\Delta Y_s^{-ij}(m))\right]ds\notag\\
&-\int_t^m(\Delta Z_s^{ij}(m))^{tr}dW_s,
\end{align}
where
\begin{align*}
F_s^{ij}(z)=&\ f^i(V_s,z+Z_s^{j}(m))-f^j(V_s,Z_s^j(m)),
\end{align*}
and
\begin{align*}
G_s^{ij}(y^{ij},y^{-ij})=&\ q^{ij}e^{-y^{ij}}-q^{ji}e^{y^{ij}}-\rho
y^{ij}+\sum_{k\neq j}q^{ik}e^{y^{ki}}-\sum_{k\neq
i}q^{jk}e^{-y^{jk}},
\end{align*}
for $z\in\mathbb{R}^d$ and $y=(y^{ij},y^{-ij})\in\mathbb{R}^{m^0}$,
with $|z|\leq 2K_z$ and $|y^{ij}|\leq 2K_y$ for $i,j\in I$.

Since $F^{ij}_s(z)$ and $G_s^{ij}(y^{ij},y^{-ij})$ are Lipschitz
continuous, following along similar arguments as in section
\ref{subsection_estimate}, we deduce that $(\Delta Y^{ij}(m),\Delta
Z^{ij}(m))_{i,j\in I}$ is the unique solution to BSDE system
(\ref{fhorizon_BSDE_system_truncation_1}). Moreover, by Assumption
\ref{assumption_11}(ii)-(iii), we have, for $v,z\in\mathbb{R}^d$,
$|f^i(v,z)|\leq K_f+C_z(|z|+|z|^2),$ so
$$F_s^{ij}(0)=f^i(V_s,Z_s^{j}(m))-f^j(V_s,Z_s^j(m))\leq 2K_f+2C_z(K_z+K_z^2).$$
Using $\sum_{k\neq j}q^{ik}=-q^{ij}$ and $\sum_{k\neq
i}q^{jk}=-q^{ji}$, we also have
\begin{align*}
G_s^{i}(\bar{Y}_s,\bar{Y}_s^{-i})=-(q^{ij}+q^{ji})(e^{\bar{Y}_s}-e^{-\bar{Y}_s})-\rho
\bar{Y}_s,
\end{align*}
where $\bar{Y}$ solves the ODE
$$\bar{Y}_t=\int_t^m2\left[K_f+C_z(K_z+K_z^2)-q^{min}\bar{Y}_s\right]ds.$$
Since $0\leq \bar{Y}_t\leq \frac{K_f+C_z(K_z+K_z^2)}{q^{min}}$, we
further have
\begin{align*}
G_s^{i}(\bar{Y}_s,\bar{Y}_s^{-i})&\leq
-(q^{ij}+q^{ji})(e^{\bar{Y}_s}-e^{-\bar{Y}_s})\\
&\leq-2q^{min}(\bar{Y}_s+1-e^{-\bar{Y}_s})\leq -2q^{\min}\bar{Y}_s,
\end{align*}
and, consequently, using Lemma \ref{comparsion_lemma} we deduce that
$$\Delta Y_t^{ij}(m)\leq \bar{Y}_t\leq
\frac{K_f+C_z(K_z+K_z^2)}{q^{min}}.$$ By the symmetric property, we
also have $\Delta Y_t^{ji}(m)\leq
\frac{K_f+C_z(K_z+K_z^2)}{q^{min}}$, from which we obtain estimate
(\ref{diff_estimate}).
\end{proof}

Next,  we send $\rho\rightarrow 0$ in the infinite horizon BSDE
system (\ref{infhorizon_BSDE_system}). To emphasize the dependencies
on $\rho$ and $V_0=v$, we use the notations $V_t^v, Y_t^{i,\rho,v}$
and $Z_t^{i,\rho,v}$ in the rest of this section. Sending
$m\rightarrow \infty$ in the estimate
(\ref{gradient_estimate}) yields that, for the first component
$Y_t^{i,\rho,v}=\mathbf{y}^{i,\rho}(V_t^{v})$ of the solution to
(\ref{infhorizon_BSDE_system}),
\begin{equation}\label{Lip_estimate}
|\mathbf{y}^{i,\rho}(v_1)-\mathbf{y}^{i,\rho}(v_2)|\leq \frac{C_{v}}{C_{\eta}-C_v}|v_1-v_2|, \quad v_1,v_2\in \mathbb{R}^d.
\end{equation}

Given a fixed reference point, say $v_0\in\mathbb{R}^d$, we define
the processes
$\bar{Y}_t^{i,\rho,v}:=Y_t^{i,\rho,v}-Y_0^{m^0,\rho,v_0}$, for
$t\geq 0$, $i\in I$ and $v\in\mathbb{R}^d$, and consider the
perturbed version of the infinite horizon BSDE system
(\ref{infhorizon_BSDE_system})\footnote{There is nothing special about the choice of the reference point $m^0$. Any regime $j\in I$ will also serve the purpose.}, i.e.
\begin{align}\label{perturbed_BSDE}
\bar{Y}_t^{i,\rho,v}=&\
\bar{Y}_T^{i,\rho,v}+\int_t^T\left[\sum_{k\in
I}q^{ik}(e^{\bar{Y}_s^{k,\rho,v}-\bar{Y}_s^{i,\rho,v}}-1)-\rho
\bar{Y}_s^{i,\rho,v}+\rho
{Y}_0^{m^{0},\rho,v_0}\right]ds\notag\\
&+\int_t^Tf^i(V_s^{v},Z_s^{i,\rho,v})ds
-\int_t^T(Z_s^{i,\rho,v})^{tr}dW_s,
\end{align}
for $0\leq t\leq T<\infty$, $i\in I$ and $v\in\mathbb{R}^d$. By the
Markov property of $Y^{i,\rho,v}$ (see Theorem 4.1 in \cite{ElKaroui1997}), we have
$\bar{Y}_t^{i,\rho,v}=\bar{\mathbf{y}}^{i,\rho}(V_t^v)$ with
$\bar{\mathbf{y}}^{i,\rho}(\cdot):=\mathbf{y}^{i,\rho}(\cdot)-\mathbf{y}^{m^0,\rho}(v_0)$.

Note that, by estimate (\ref{Lip_estimate}),
${\mathbf{y}}^{i,\rho}(\cdot)$ is Lipschitz continuous uniformly in
$\rho$, and by estimate (\ref{diff_estimate_0}),
$\bar{\mathbf{y}}^{i,\rho}(v_0)=\mathbf{y}^{i,\rho}(v_0)-\mathbf{y}^{m^0,\rho}(v_0)$
is bounded uniformly in $\rho$. In turn, we deduce that, for
$v\in\mathbb{R}^d$,
\begin{align}\label{estimate_y_rho}
|\bar{\mathbf{y}}^{i,\rho}(v)|&=|\mathbf{y}^{i,\rho}(v)-\mathbf{y}^{i,\rho}(v_0)+\mathbf{y}^{i,\rho}(v_0)-\mathbf{y}^{m^0,\rho}(v_0)|\notag\\
&\leq
\frac{C_{v}}{C_{\eta}-C_v}|v-v_0|+\frac{1}{q^{\min}}\left(K_f+\frac{C_vC_{\eta}C_z}{(C_{\eta}-C_v)^2}\right).
\end{align}
Moreover, (\ref{estimate_yz}) implies that $|\rho
\mathbf{y}^{m^0,\rho}(v_0)|\leq \rho K_y=K_f$. Hence, by a standard
diagonal procedure, there exists a sequence, denoted by
$\{\rho_n\}_{n\geq 1}$, such that, for $v$ in a dense subset of
$\mathbb{R}^d$,
$$\lim_{\rho_n\rightarrow 0}\rho_n\mathbf{y}^{m^0,\rho_n}(v_0)=\lambda,\quad
\lim_{\rho_n\rightarrow
0}\bar{\mathbf{y}}^{i,\rho_n}(v)=\mathbf{y}^i(v),$$ for some
$\lambda\in\mathbb{R}$ and the limit function ${\mathbf{y}}^i(v)$.

Since $\bar{\mathbf{y}}^{i,\rho}(\cdot)$ is Lipschitz continuous
uniformly in $\rho$, the limit function $\mathbf{y}^i(\cdot)$ can be
further extended to a Lipschitz continuous function defined for all
$v\in\mathbb{R}^d$, i.e. for $v\in\mathbb{R}^d$,
$$\lim_{\rho_n\rightarrow 0}\bar{\mathbf{y}}^{i,\rho_n}(v)=\mathbf{y}^i(v).$$
Thus, for the infinite horizon BSDE system (\ref{perturbed_BSDE}),
it holds that $\lim_{\rho_n\rightarrow
0}\bar{Y}_t^{i,\rho_n,v}=\mathbf{y}^{i}(V_t^v)$ and
$\lim_{\rho_n\rightarrow 0}\rho_n \bar{Y}_t^{i,\rho_n,v}=0$.

As a result, by defining the processes
$\mathcal{Y}^{i,v}_t:=\mathbf{y}^{i}(V_t^v)$, for $t\geq 0$, $i\in
I$ and $v\in\mathbb{R}^d$, it is standard to show that (see
\cite{HU1} and \cite{HU2}) there exist a limit function $\mathbf{z}^i(\cdot)$ such that $Z^{i,\rho_n,v}$ converges to
$\mathcal{Z}^{i,v}:=\mathbf{z}^i(V_t^v)\in \mathcal{L}^2$ as $\rho_n\to 0$, and
$\left((\mathcal{Y}^{i,v},\mathcal{Z}^{i,v})_{i\in
I},\lambda\right)$ solve the ergodic BSDE system
\begin{equation}\label{ergodic_BSDE_system}
d\mathcal{Y}_t^{i,v}=-f^i(V_t^{v},\mathcal{Z}_t^{i,v})dt-\sum_{k\in
I}q^{ik}(e^{\mathcal{Y}_t^{k,v}-\mathcal{Y}_t^{i,v}}-1)dt+\lambda dt
+(\mathcal{Z}_t^{i,v})^{tr}dW_t,
\end{equation}
for $t\geq 0$, $i\in I$ and $v\in\mathbb{R}^d$.

The main result of this section is  the following existence and uniqueness of
the solution to the ergodic BSDE system (\ref{ergodic_BSDE_system}). Clearly,  ergodic BSDE~(\ref{ergodic_BSDE_system}) admits multiple ( possibly non-Markovian) solutions.  It is more reasonable to  consider the uniqueness of functions rather than processes (see Remark 4.7 in \cite{HU2}). This explains why we are only concerned with Markovian solutions of (\ref{ergodic_BSDE_system}) in the rest of the paper.

\begin{theorem}\label{theorem_ergodic_BSDE_system} Suppose that Assumptions
\ref{assumption_11}-\ref{assumption_15} are satisfied. Then, there
exists a unique Markovian solution
$((\mathcal{Y}^{i,v}_t,\mathcal{Z}^{i,v}_t)_{i\in
I},\lambda)=((\mathbf{y}^i(V_t^v),\mathbf{z}^i(V_t^v))_{i\in I},\lambda)$, $t\geq 0$, to the ergodic BSDE system
(\ref{ergodic_BSDE_system}), such that the  functions $(\mathbf{y}^i(\cdot),\mathbf{z}^i(\cdot))$ satisfy
\begin{align}
|\mathbf{y}^{i}(v)|&\leq
C_y(1+|v|),\label{estmate_y_1}\\
|\mathbf{z}^{i}(v)|&\leq
K_z=\frac{C_v}{C_{\eta}-C_v}\label{estimate_z_1},\\
|\mathbf{y}^{i}(v)-\mathbf{y}^{j}(v)|&\leq
\frac{1}{q^{\min}}\left(K_f+\frac{C_vC_{\eta}C_z}{(C_{\eta}-C_v)^2}\right)\label{estimate_y_2},
\end{align}
for some constant $C_y>0$, where all other constants are given
in Lemma \ref{lemma_difference}. The function $\mathbf{y}^i(\cdot)$
is unique up to an additive constant and, without loss of
generality, it is set that $\mathbf{y}^i(0)=0$.
\end{theorem}

\begin{proof} We have already shown the existence of a Markovian solution to
(\ref{ergodic_BSDE_system}). The estimates (\ref{estmate_y_1}), (\ref{estmate_y_1}),  and (\ref{estimate_y_2})
follow, respectively, from
(\ref{estimate_y_rho}),  (\ref{estimate_yz}), and (\ref{diff_estimate_0}) by sending
$\rho\rightarrow 0$. Hence, it remains  to show the uniqueness. The
idea is to convert the ergodic BSDE system
(\ref{ergodic_BSDE_system}) to a scalar-valued ergodic BSDE driven
by the Brownian motion $W$ and an exogenously given Markov chain
$\alpha$. We postpone this part of the proof to Appendix
\ref{Appendix_B} after we introduce the Markov chain $\alpha$ in the
next section.
\end{proof}

\begin{remark} The conditions (\ref{estmate_y_1})-(\ref{estimate_y_2}) are essential for the uniqueness of the Markovian solution to (\ref{ergodic_BSDE_system}). We provide examples of Markovian solutions which do not satisfy them. Assume that $d=m^0=1$, $\eta(v)=-\frac12 v$, and $\kappa=1$. Then,  (\ref{ergodic_BSDE_system}) reduces to
$$d\mathcal{Y}_t^v=-f(V_t^v,\mathcal{Z}_t^v)dt+\lambda dt+\mathcal{Z}^vdW_t,$$
with $dV_t^v=-\frac12V_t^vdt+dW_t$ and $V_0^v=v$.

As the first example, we consider $f(v,z)=\frac{v}{2}e^{-v^2/2}$.  Assumptions 1-4 are then all satisfied. The unique Markovian solution satisfying (\ref{estmate_y_1})-(\ref{estimate_y_2}) is given by $(\mathcal{Y}^{v}_t,\mathcal{Z}^{v}_t,\lambda)=(\mathbf{y}(V_t^v),\mathbf{z}(V_t^v),0)$ with $$(\mathbf{y}(v),\mathbf{z}(v))=\left(\frac12\int_{-\infty}^{v}e^{-\frac{y^2}{2}}dy,\frac12e^{-\frac{v^2}{2}}\right).$$
It is easy to check that both triplets  $$
\left(\frac12\int_{0}^{v}[e^{-\frac{y^2}{2}}-e^{\frac{y^2}{2}}]dy,\frac12[e^{-\frac{v^2}{2}}-e^{\frac{v^2}{2}}],0\right)$$
and $$\left(\int_0^{v}e^{\frac{y^2}{2}}[\frac12e^{-y^2}+N(y)-1]dy,e^{\frac{v^2}{2}}[\frac12e^{-v^2}+N(v)-1],\frac{1}{2\sqrt{2\pi}}\right),$$ where  $N(x):=\frac{1}{\sqrt{2\pi}}\int_{-\infty}^{x}e^{-\frac{y^2}{2}}dy$, also satisfy (\ref{ergodic_BSDE_system}). However, neither of them satisfies the conditions (\ref{estmate_y_1}) and (\ref{estimate_z_1}).

As the second example, we consider $f(v,z)=\frac{|v|}{2}e^{-v^2/2}$. The unique Markovian solution satisfying (\ref{estmate_y_1})-(\ref{estimate_y_2}) is given by the triplet $(\mathbf{y}(\cdot),\mathbf{z}(\cdot),\frac{1}{2\sqrt{2\pi}})$ with
\begin{align*}
\mathbf{y}(v)&=\chi_{\{v\geq 0\}}\int_0^{v}e^{\frac{y^2}{2}}[\frac12e^{-y^2}+N(y)-1]dy
+\chi_{\{v< 0\}}\int_0^{v}e^{\frac{y^2}{2}}[-\frac12e^{-y^2}+N(y)]dy;\\
\mathbf{z}(v)&=\chi_{\{v\geq 0\}}e^{\frac{v^2}{2}}[\frac12e^{-v^2}+N(v)-1]+\chi_{\{v< 0\}}e^{\frac{v^2}{2}}[-\frac12e^{-v^2}+N(v)].
\end{align*}
However, it is easy to check that the triplet $(\bar{\mathbf{y}}(\cdot),\bar{\mathbf{z}}(\cdot),0)$ with
\begin{align*}
\bar{\mathbf{y}}(v)&=\chi_{\{v\geq 0\}}\int_0^{v}(\frac12{e^{\frac{-y^2}{2}}}-e^{\frac{y^2}{2}})dy
-\chi_{\{v<0\}}\int_0^{v}\frac12{e^{-\frac{y^2}{2}}}dy;\\
\bar{\mathbf{z}}(v)&=\chi_{\{v\geq 0\}}(\frac12{e^{\frac{-v^2}{2}}}-e^{\frac{v^2}{2}})
-\chi_{\{v<0\}}\frac12{e^{-\frac{v^2}{2}}},
\end{align*}
also satisfies (\ref{ergodic_BSDE_system}),  but fails to satisfy the conditions (\ref{estmate_y_1}) and (\ref{estimate_z_1}).
\end{remark}

\section{Application to regime switching forward
performance processes}\label{section:regime_switching_forward}


Let $(\Omega,\mathcal{F},\mathbb{F},\mathbb{P})$ be the filtered
probability space introduced in section \ref{section:BSDE}. Assume
the probability space also supports a Markov chain $\alpha$ with its
augmented filtration $\mathbb{H}=\{\mathcal{H}_t\}_{t\geq 0}$
independent of the Brownian filtration $\mathbb{F}$. The Markov
chain $\alpha$ has the transition rate matrix $\mathcal{Q}$ as
specified in Assumption \ref{assumption_12}, and admits the
representation
$$d\alpha_t=\sum_{k,k'\in
I}(k-k')\chi_{\{\alpha_{t-}=k'\}}d{N}_t^{k'k},$$
where $(N^{k^{\prime}k})_{k^{\prime},k\in I}$ are independent
Poisson processes each with intensity $q^{k^{\prime}k}$ (see chapter 9.1.2 in \cite{Bremaud}). Let $T_0=0$ and $T_1,T_2,\dots$ be the jump times of the Markov
chain $\alpha$, and $(\alpha^{j})_{j\geq 1}$ be a sequence of
$\mathcal{H}_{T_{j}}$-measurable random variables representing the
position of $\alpha$ in the time interval $[T_{j-1},T_{j})$. Hence, $\alpha_t=\sum_{j\geq 1}\alpha^{j-1}\chi_{[T_{j-1},T_{j})}(t)$.
Without
loss of generality, assume that $\alpha^0=i\in I$. Denote the
smallest filtration generated by $\mathbb{F}$ and $\mathbb{H}$ as
$\mathbb{G}=\{\mathcal{G}_t\}_{t\geq 0}$, i.e.
$\mathcal{G}_t=\mathcal{F}_t\vee\mathcal{H}_t$.

 We consider a market consisting of a risk-free bond offering zero interest
rate and $n$ risky assets, with $n\leq d$. The prices of the $n$
risky assets are driven by the Markov chain $\alpha$ and a
$d$-dimensional stochastic factor process $V$, which satisfies
Assumption \ref{assumption_13}.

Each state $i\in I$ of the Markov chain $\alpha$ represents a market
regime, and in regime $i$, the corresponding market price of risk at
time $t$ is $\theta^i(V_t)$. The $n$-dimensional price process
$S=(S^1,\dots,S^n)^{tr}$ of the risky assets follows
\begin{equation}
dS_t=\text{diag}(S_t)\sigma(V_{t})(\theta^{\alpha_{t-}}(V_t)dt+dW_t)\text{,}
\label{stock}
\end{equation}%
where $\sigma(V_t)\in\mathbb{R}^{n\times d}_+$ is the volatility
matrix of the risky assets at time $t$, and
$\text{diag}(S_t)=\{\text{diag}(S_t)_{kj}\}_{1\leq k,j\leq n}$, with
$\text{diag}(S_t)_{kk}=S^k_{t}$ and $\text{diag}(S_t)_{kj}=0$ for
$k\neq j$, represents the prices of the risky assets at time $t$.

\begin{assumption}\label{assumption_14} The market coefficients of
the $n$ risky assets satisfy that

(i) $\sigma(v)$ is uniformly bounded in $v\in\mathbb{R}^d$ and has full rank
  $n$;

 (ii) for $i\in I$, $\theta^{i}(v)$ is uniformly bounded and Lipschitz continous
  in $v\in\mathbb{R}^d$.
\end{assumption}

\begin{remark}\label{remark_4}
{When $n<d$, the financial market is  incomplete. A typical example is $n=1$ and $d=2$ for the following regime switching stochastic volatility model
\begin{align*}
dS_{t}&=S_{t}\sigma(V_t)(\theta^{\alpha_{t-}}(V_t)dt+dW_{t}), \\
dV_{t}&=\eta \left( V_{t}\right) dt+\kappa\, dW_{t}.
\end{align*}%
Here,  the function $\sigma (\cdot)$ takes values in a two-dimensional row vector  space, all the $m^0+1$ functions $\theta ^i(\cdot), i\in I, $   and $\eta(\cdot)$  take values in a two-dimensional column vector space, and the constant matrix $\kappa \in \mathbb{R}^{2\times 2}$
is positive definite and  normalized to $|\kappa |=1$.}
\end{remark}

\subsection{Trading strategies}
In this market environment, an investor trades dynamically among the
risk-free bond and the risky assets. Let $\tilde{\pi}=(\tilde{\pi}%
^{1},\dots ,\tilde{\pi}^{n})^{tr}$ denote the (discounted by the
bond) proportions of her wealth in the risky assets. They are taken
to be self-financing and, thus, the (discounted by the bond) wealth
process satisfies
\begin{equation*}
dX_{t}(\tilde{\pi})=X_t(\tilde{\pi})\tilde{\pi}_{t}^{tr}\sigma
(V_{t})\left( \theta^{\alpha_{t-}}(V_{t})dt+dW_{t}\right).
\end{equation*}%
As in \cite{LZ}, we work with the trading strategies rescaled by the
volatility matrix, namely, $ \pi
_{t}^{tr}:=\tilde{\pi}_{t}^{tr}\sigma (V_{t}). $ Then, the wealth
process in regime $i$ satisfies
\begin{equation}
dX_{t}(\pi)=X_t(\pi)\pi _{t}^{tr}\left( \theta^{\alpha_{t-}}
(V_{t})dt+dW_{t}\right) . \label{wealth}
\end{equation}%

For any $t\geq 0$, we denote by $\mathcal{A}^{\mathbb{G}}_{[0,t]}$
the set of admissible trading strategies in $[0,t],$ defined as
\begin{eqnarray*}
\mathcal{A}_{[0,t]}^{\mathbb{G}}:=&&\left \{ \pi
_{s}=\pi_0^i\chi_{\{0\}}(s)+\sum_{j\geq
1}\pi_s^{\alpha^{j-1}}\chi_{(T_{j-1},T_{j}]}(s),\
s\in[0,t]:\ \pi_s^{j}\in\Pi^j,\right.\\
&& \left.\pi^j\ \text{is}\ \mathbb{F}\text{-progressively\
measurable and}\int_0^{t}|\pi_s^j|^{2}ds<\infty, \mathbb{P}\text{-a.s.}\right\},
\end{eqnarray*}%
where $\Pi^j$, $j\in I$, are closed and convex subsets in
$\mathbb{R}^{d}$. So $\Pi^j$ models the investor's trading
constraints, and the investor will adjust her trading constraint
sets according to different
market regimes. 

For $0\leq t\leq s$, the set $\mathcal{A}^{\mathbb{G}}_{[t,s]}$ is
defined in a similar way, and the set of admissible trading
strategies for {all} $t\geq 0$ is, in turn, defined as
$\mathcal{A}^{\mathbb{G}}=\cup _{t\geq
0}\mathcal{A}^{\mathbb{G}}_{[0,t]}$.

For the regime switching stochastic volatility model in Remark \ref{remark_4}, a typical choice of the trading constraint set $\Pi^j$ is $\Pi^j=\mathbb{R}\times\{0\}$ for $j\in I$.

\subsection{Regime switching forward performance processes}

The investor uses a forward criterion to measure the performance for
her admissible trading strategies. We introduce the definition of
regime switching forward performance processes associated with this
market.

\begin{definition}\label{def:forward_performance}
A family of stochastic processes $\left(U^i(x,t)\right)_{i\in I}$,
for $\left( x,t\right) \in \mathbb{R}_+^2$, is a regime switching
forward performance process if the following conditions are
satisfied:

(i) For each $i\in I$ and $x\in \mathbb{R}_+$, $t\mapsto U^i\left( x,t\right) $ is $\mathbb{F}$%
-progressively measurable;

(ii) For each $i\in I$ and $t\geq 0$, the mapping $x\mapsto
U^i(x,t)$ is strictly increasing and strictly concave;

(iii) Define the process
\begin{equation}\label{regime_switching_forward}
U(x,t):=\sum_{j\geq
1}U^{\alpha^{j-1}}(x,t)\chi_{[T_{j-1},T_{j})}(t).
\end{equation}
Then, for all $\pi\in \mathcal{A}^{\mathbb{G}}$ and $0\leq t\leq s$,
\begin{equation}\label{supermartingale}
U( X_{t}({\pi}),t) \geq \mathbb{E}\left[ U(X_{s}({\pi}),s)|%
\mathcal{G}_{t}\right] ,
\end{equation}%
and there exists an optimal $\pi^{\ast }\in
\mathcal{A}^{\mathbb{G}}$ such that
\begin{equation}\label{martingale}
U( X_{t}({\pi ^{\ast }}),t) =\mathbb{E}\left[ U(X_{s}({\pi ^{\ast
}}),s)|\mathcal{G}_{t}\right] ,
\end{equation}%
with $X({\pi}),X(\pi^{\ast})$ solving (\ref{wealth}).
\end{definition}

The above (super)martingale conditions (\ref{supermartingale}) and (\ref{martingale}) can be restated as follows: For
$j\geq 1$, on the event $\{T_{j-1}\leq t< T_{j}\}$,
\begin{align*}
U(x,t)=U^{\alpha^{j-1}}(x,t)=\esssup_{\pi\in\mathcal{A}^{\mathbb{G}}_{[t,s]}}
&\ \mathbb{E}\left[U^{\alpha^{j-1}}(X_s(\pi),s)\chi_{\{s<T_{j}\}}\right.\\
&\left.+ U^{\alpha^{j}}(X_{T_{j}}(\pi),T_{j})\chi_{\{s\geq
T_{j}\}}|\mathcal{F}_t,X_t=x\right],
\end{align*}
and on $\{t=T_{j}\}$, $U(x,t)$ has a jump with size
$$U(x,T_{j})-U(x,T_{j}-)=U^{\alpha^{j}}(x,T_j)-U^{\alpha^{j-1}}(x,T_j-).$$
Hence, we have the following decomposition formula for $U(x,t)$
(recall that $\alpha^0=i$):
\begin{align}\label{decompose_formula}
U(x,t)=&\ U(x,0)+\sum_{j\geq 1}\left[U(x,t\wedge T_{j}-)-U(x,t\wedge
T_{j-1})\right]\notag\\
&\ +\sum_{j\geq 1}\left[U(x,t\wedge T_j)-U(x,t\wedge
T_{j}-)\right]\notag\\
=&\ U^{i}(x,0)+\sum_{j\geq 1}\left[U^{\alpha^{j-1}}(x,t\wedge
T_{j}-)-U^{\alpha^{j-1}}(x,t\wedge T_{j-1})\right]\notag\\
&\ +\sum_{j\geq 1}\left[U^{\alpha^j}(x,
T_j)-U^{\alpha^{j-1}}(x,T_{j}-)\right]\chi_{\{T_j\leq t\}}.
\end{align}
The first sum on the right hand side of (\ref{decompose_formula}) is
the continuous component of $U(x,t)$, while the second sum is the
jump component of $U(x,t)$.

Here, we focus on \emph{Markovian regime switching forward performance processes in power form}, namely, the processes that are \emph{deterministic} functions of the stochastic factor process $V$,
$$U^i(x,t)=\frac{x^{\delta}}{\delta}e^{K^{i}(V_t,t)}$$
for $\delta\in (0, 1)$ and appropriate function(s) $K^i:\mathbb{R}^d\times\mathbb{R}_+\rightarrow\mathbb{R}$.

\subsection{Representation via system of ergodic BSDE}

We now characterize
Markovian regime switching forward performance processes via the ergodic BSDE
system (\ref{ergodic_BSDE_system}) introduced in Section 3. For
$i\in I$ and $(v,z)\in\mathbb{R}^d\times\mathbb{R}^d$, we consider
the driver
\begin{equation}\label{driver_BSDE}
f^i(v,z)=\frac12\delta(\delta-1)\text{dist}^2\left(\Pi,\frac{z+\theta^i(v)}{1-\delta}\right)
+\frac{\delta}{2(1-\delta)}|z+\theta^i(v)|^2+\frac{|z|^2}{2}.
\end{equation}
It is easy to check that $f^i$ satisfies Assumption
\ref{assumption_11}.  Then,  from Theorem
\ref{theorem_ergodic_BSDE_system}, the ergodic BSDE system
(\ref{ergodic_BSDE_system}) admits a unique Markovian solution
$\left((\mathcal{Y}^{i},\mathcal{Z}^{i})_{i\in I},\lambda\right)$
satisfying (\ref{estmate_y_1}), (\ref{estimate_z_1}) and (\ref{estimate_y_2}).

\begin{theorem}\label{theorem_ergodic_representation} Suppose that Assumptions \ref{assumption_11}-\ref{assumption_14} are satisfied. Let
$((\mathcal{Y}_t^{i},\mathcal{Z}_t^{i})_{i\in I},\lambda)=((\mathbf{y}^i(V_t^v),\mathbf{z}^i(V_t^v))_{i\in I},\lambda)$, $t\geq 0$, be
the unique Markovian solution of the ergodic BSDE system
(\ref{ergodic_BSDE_system}) with driver $f^i$ as in
(\ref{driver_BSDE}), and satisfy (\ref{estmate_y_1}),
(\ref{estimate_z_1}) and (\ref{estimate_y_2}). Then,
\begin{equation}\label{representation_2}
U^i(x,t)=\frac{x^{\delta}}{\delta}e^{\mathcal{Y}_t^{i}-\lambda t}=\frac{x^{\delta}}{\delta}e^{\mathbf{y}^i(V_t^v)-\lambda t},\
i\in I,
\end{equation}
form a Markovian regime switching forward performance process, and in each
regime $i$,
\begin{equation}\label{optimal_strategy}
\pi_t^{i,\ast}=\text{Proj}_{\Pi^i}\left(\frac{\mathcal{Z}_t^{i}+\theta^{i}(V_t)}{1-\delta}\right)
\end{equation}
is the associated optimal trading strategy in this regime.
\end{theorem}



\begin{remark}\label{remark}
The boundedness conditions (\ref{estimate_z_1}) and (\ref{estimate_y_2}) are crucial for the verification of the (super)martingale conditions of $U(x,t)$ (see step 3 in section \ref{section:proof}), while the linear growth condition (\ref{estmate_y_1}) is used to connect forward performance processes and classical utility maximization (see Proposition \ref{propositionLambda}).

In particular, if there is only a single regime, i.e. $m^0=1$, then the ergodic BSDE system
(\ref{ergodic_BSDE_system}) reduces to
$$d\mathcal{Y}_t^1=-f^1(V_t,\mathcal{Z}_t^1)dt+\lambda dt+(\mathcal{Z}_t^1)^{tr}dW_t.$$
In this case, the Markovian forward performance process has the representation
$$U^1(x,t)=\frac{x^{\delta}}{\delta}e^{\mathcal{Y}_t^1-\lambda t}=\frac{x^{\delta}}{\delta}e^{\mathbf{y}^1(V_t^v)-\lambda t},$$
which is precisely the representation formula established in
\cite[Theorem 3.2 ]{LZ}.
\end{remark}

To prove Theorem \ref{theorem_ergodic_representation}, we need
It\^o's formula for the Markov chain $\alpha$. We recall it in
the following lemma, which will be frequently used in the rest of
the paper. Its proof is a straightforward extension of
\cite{Bremaud} and \cite{Zhang1} and is thus omitted here.

\begin{lemma}\label{lemma_Ito}
For $i\in I$, let $F^i_t$, $t\geq 0$, be a family of
$\mathbb{F}$-progressively measurable and continuous stochastic
processes. Then,
\begin{align*}
&\sum_{j\geq
1}\left[F^{\alpha_{T_j}}_{T_j}-F^{\alpha_{T_{j}-}}_{T_{j}-}\right]\chi_{\{T_{j}\leq
t\}}\\
=&\int_{0}^{t}\sum_{k\in
I}q^{\alpha_{s-}k}[F^k_s-F^{\alpha_{s-}}_s]ds+\int_{0}^{t}\sum_{k,k^{\prime}\in
I}[F_s^k-F_s^{k'}]\chi_{\{\alpha_{s-}=k^{\prime}\}}d\tilde{N}_s^{k^{\prime}k},
\end{align*}
where
$\tilde{N}_t^{k^{\prime}k}=N_t^{k^{\prime}k}-q^{k^{\prime}k}t$,
$t\geq 0$, are the compensated Poisson martingales
under the filtration $\mathbb{G}=\mathbb{F}\vee\mathbb{H}$.
\end{lemma}

\subsection{Proof of Theorem \ref{theorem_ergodic_representation}}\label{section:proof}
We divide the proof into three steps. The first two steps derive,
locally and globally, the stochastic dynamics of the regime
switching forward performance process. The last step verifies the super(martingale) conditions in Definition \ref{def:forward_performance}.


\emph{Step 1}. For $t\geq 0$ and $i\in I$, let
$\bar{\mathcal{Y}}_t^i:=\mathcal{Y}_t^i-\lambda t$. Then, in each
time interval $[T_{j-1},T_j)$, we have
$$U(x,t)=U^{\alpha^{j-1}}(x,t)=\frac{x^{\delta}}{\delta}e^{\bar{\mathcal{Y}}_t^{\alpha^{j-1}}}.$$
On the other hand, for $t\in(T_{j-1},T_j]$, note that any admissible
trading strategy $\pi\in\mathcal{A}^{\mathbb{G}}$ takes the form
$\pi_t=\pi^{\alpha^{j-1}}_t,$ with $\pi^{\alpha^{j-1}}$ being
$\mathbb{F}$-progressively measurable. In turn, applying It\^o's
formula and using the equations (\ref{infhorizon_BSDE_system}) and
(\ref{wealth}), we obtain
\begin{align}\label{Ito_small_interval}
&\frac{(X_{T_{j}-}(\pi))^{\delta}}{\delta}e^{\bar{\mathcal{Y}}_{T_{j}-}^{\alpha^{j-1}}}-
\frac{(X_{T_{j-1}}(\pi))^{\delta}}{\delta}e^{\bar{\mathcal{Y}}_{T_{j-1}}^{\alpha^{j-1}}}\notag\\
=\
&\int_{T_{j-1}}^{T_{j}}\frac{(X_s(\pi))^{\delta}}{\delta}e^{\bar{\mathcal{Y}}_s^{\alpha^{j-1}}}
\left[f^{\alpha^{j-1}}(V_s,\mathcal{Z}_s^{\alpha^{j-1}};\pi^{\alpha^{j-1}}_s)-
f^{\alpha^{j-1}}(V_s,\mathcal{Z}_s^{\alpha^{j-1}})\right]ds\notag\\
&\
+\int_{T_{j-1}}^{T_{j}}\frac{(X_s(\pi))^{\delta}}{\delta}e^{\bar{\mathcal{Y}}_s^{\alpha^{j-1}}}
\sum_{k\in
I}q^{\alpha^{j-1}k}\left[1-e^{\bar{\mathcal{Y}}_s^{k}-\bar{\mathcal{Y}}_s^{\alpha^{j-1}}}\right]
ds\notag\\
&\
+\int_{T_{j-1}}^{T_{j}}\frac{(X_s(\pi))^{\delta}}{\delta}e^{\bar{\mathcal{Y}}_s^{\alpha^{j-1}}}
\left(\delta\pi^{\alpha^{j-1}}_s+\mathcal{Z}_s^{\alpha^{j-1}}\right)^{tr}dW_s,
\end{align}
where
\begin{equation}\label{driver_pi}
f^{i}(v,z;\pi):=\frac12\delta(\delta-1)|\pi|^2+\delta\pi^{tr}\theta^i(v)+\delta\pi^{tr}z+\frac12|z|^2,
\end{equation}
for $i\in I$ and
$(v,z,\pi)\in\mathbb{R}^d\times\mathbb{R}^d\times\mathbb{R}^d$.

\emph{Step 2}.
Then, for $t\geq 0$ and $\pi\in\mathcal{A}^{\mathbb{G}}$, i.e.
$\pi_t=\pi_0^i+\sum_{j\geq
1}\pi^{\alpha^{j-1}}_t\chi_{(T_{j-1},T_j]}(t),$ using the
decomposition formula (\ref{decompose_formula}), we further have
\begin{align*}
\frac{(X_t(\pi))^{\delta}}{\delta}e^{\bar{\mathcal{Y}}_t^{\alpha_t}}-\frac{x^{\delta}}{\delta}e^{\bar{\mathcal{Y}}_0^{i}}
=&\ \sum_{j\geq 1}\left[\frac{(X_{t\wedge
T_{j}-}(\pi))^{\delta}}{\delta}e^{\bar{\mathcal{Y}}_{t\wedge
T_{j}-}^{\alpha^{j-1}}}-
\frac{(X_{t\wedge T_{j-1}}(\pi))^{\delta}}{\delta}e^{\bar{\mathcal{Y}}_{t\wedge T_{j-1}}^{\alpha^{j-1}}}\right]\\
&\ +\sum_{j\geq 1}\left[\frac{(X_{
T_{j}}(\pi))^{\delta}}{\delta}e^{\bar{\mathcal{Y}}_{
T_{j}}^{\alpha^{j}}}-\frac{(X_{
T_{j}-}(\pi))^{\delta}}{\delta}e^{\bar{\mathcal{Y}}_{
T_{j}-}^{\alpha^{j-1}}}\right]\chi_{\{T_{j}\leq
t\}}\\
=&\ (I)+(II).
\end{align*}

For the continuous component $(I)$, using (\ref{Ito_small_interval})
and the facts that $\alpha_{s-}=\alpha^{j-1}$,
$\pi_s=\pi^{\alpha^{j-1}}_s$, for $s\in(t\wedge T_{j-1},t\wedge
T_j]$, we deduce that
\begin{align}\label{continous_component}
(I)=&\int_{0}^t\frac{(X_s(\pi))^{\delta}}{\delta}e^{\bar{\mathcal{Y}}_s^{\alpha_{s-}}}\left[f^{\alpha_{s-}}(V_s,\mathcal{Z}_s^{\alpha_{s-}};\pi_s)-f^{\alpha_{s-}}(V_s,\mathcal{Z}_s^{\alpha_{s-}})
\right]ds\notag\\
&+\int_0^t\frac{(X_s(\pi))^{\delta}}{\delta}e^{\bar{\mathcal{Y}}_s^{\alpha_{s-}}}\sum_{k\in
I}q^{\alpha_{s-}k}\left[1-e^{\bar{\mathcal{Y}}_s^{k}-\bar{\mathcal{Y}}_s^{\alpha_{s-}}}\right]ds\notag\\
&+\int_0^t\frac{(X_s(\pi))^{\delta}}{\delta}e^{\bar{\mathcal{Y}}_s^{\alpha_{s-}}}\left(\delta\pi_s+\mathcal{Z}_s^{\alpha_{s-}}\right)^{tr}dW_s.
\end{align}

For the jump component $(II)$, using Lemma \ref{lemma_Ito}, we
deduce that
\begin{align}\label{jump_component}
(II)=&\int_0^t\frac{(X_s(\pi))^{\delta}}{\delta}e^{\bar{\mathcal{Y}}_s^{\alpha_{s-}}}
\sum_{k,k^{\prime}\in
I}\left[e^{\bar{\mathcal{Y}}_s^k-\bar{\mathcal{Y}}_s^{k^{\prime}}}-1\right]\chi_{\{\alpha_{s-}=k^{\prime}\}}d\tilde{N}_s^{k^{\prime}k}\notag\\
&+\int_0^t\frac{(X_s(\pi))^{\delta}}{\delta}e^{\bar{\mathcal{Y}}_s^{\alpha_{s-}}}
\sum_{k\in
I}q^{\alpha_{s-}k}\left[e^{\bar{\mathcal{Y}}_s^k-\bar{\mathcal{Y}}_s^{\alpha_{s-}}}-1\right]ds.
\end{align}

It then follows from (\ref{continous_component}) and
(\ref{jump_component}) that
\begin{align*}
\frac{(X_t(\pi))^{\delta}}{\delta}e^{\bar{\mathcal{Y}}_t^{\alpha_t}}-\frac{x^{\delta}}{\delta}e^{\bar{\mathcal{Y}}_0^{i}}
=&\int_{0}^t\frac{(X_s(\pi))^{\delta}}{\delta}e^{\bar{\mathcal{Y}}_s^{\alpha_{s-}}}\left[f^{\alpha_{s-}}(V_s,\mathcal{Z}_s^{\alpha_{s-}};\pi_s)-f^{\alpha_{s-}}(V_s,\mathcal{Z}_s^{\alpha_{s-}})
\right]ds\notag\\
&+\int_0^t\frac{(X_s(\pi))^{\delta}}{\delta}e^{\bar{\mathcal{Y}}_s^{\alpha_{s-}}}\left(\delta\pi_s+\mathcal{Z}_s^{\alpha_{s-}}\right)^{tr}dW_s\\
&+\int_0^t\frac{(X_s(\pi))^{\delta}}{\delta}e^{\bar{\mathcal{Y}}_s^{\alpha_{s-}}}
\sum_{k,k^{\prime}\in
I}\left[e^{\bar{\mathcal{Y}}_s^k-\bar{\mathcal{Y}}_s^{k^{\prime}}}-1\right]\chi_{\{\alpha_{s-}=k^{\prime}\}}d\tilde{N}_s^{k^{\prime}k}.
\end{align*}
In turn,
\begin{align}\label{stochastic_exponential}
\frac{(X_t(\pi))^{\delta}}{\delta}e^{\bar{\mathcal{Y}}_t^{\alpha_t}}=&\
\frac{x^{\delta}}{\delta}e^{\mathcal{Y}_0^{i}}\times
e^{\int_0^{t}f^{\alpha_{s-}}(V_s,\mathcal{Z}_s^{\alpha_{s-}};\pi_s)-f^{\alpha_{s-}}(V_s,\mathcal{Z}_s^{\alpha_{s-}})ds}\notag\\
&\times\mathcal{E}_t\left(\int_0^{\cdot}\left(\delta\pi_s+\mathcal{Z}_s^{\alpha_{s-}}\right)^{tr}dW_s\right)\notag\\
&\times\mathcal{E}_t\left(\int_0^{\cdot} \sum_{k,k^{\prime}\in
I}\left[e^{\mathcal{Y}_s^k-\mathcal{Y}_s^{k^{\prime}}}-1\right]\chi_{\{\alpha_{s-}=k^{\prime}\}}d\tilde{N}_s^{k^{\prime}k}\right),
\end{align}
for any $\pi\in\mathcal{A}^{\mathbb{G}}$, where $\mathcal{E}(\cdot)$
denotes Dol\'eans-Dade stochastic exponential.

\emph{Step 3}. We verify the conditions in Definition
\ref{def:forward_performance}. It is clear that (i) and (ii) hold, so we only verify the super(martingale) conditions in (iii). It follows from (\ref{driver_BSDE})
and (\ref{driver_pi}) that
$$f^{\alpha_{s-}}(V_s,\mathcal{Z}_s^{\alpha_{s-}};\pi_s)-f^{\alpha_{s-}}(V_s,\mathcal{Z}_s^{\alpha_{s-}})\leq
0,$$ for any $\pi\in\mathcal{A}^{\mathbb{G}}$. So the process $\frac{(X_t(\pi))^{\delta}}{\delta}e^{\bar{\mathcal{Y}}_t^{\alpha_t}}$, $t\geq 0$, is a local super-martingale (see (\ref{stochastic_exponential})). Next, we verify the second stochastic exponential on the right hand side of (\ref{stochastic_exponential}) is a nonnegative bounded $\mathbb{G}$-martingale. Indeed, define
$$\eta^{k'k}_s:=\left[e^{\bar{\mathcal{Y}}_s^k-\bar{\mathcal{Y}}_s^{k^{\prime}}}-1\right]\chi_{\{\alpha_{s-}=k^{\prime}\}}=
\left[e^{{\mathcal{Y}}_s^k-{\mathcal{Y}}_s^{k^{\prime}}}-1\right]\chi_{\{\alpha_{s-}=k^{\prime}\}}, \quad s\geq 0.$$
In turn,
\begin{align*}
&\mathcal{E}_t\left(\int_0^{\cdot} \sum_{k,k^{\prime}\in
I}\left[e^{\mathcal{Y}_s^k-\mathcal{Y}_s^{k^{\prime}}}-1\right]\chi_{\{\alpha_{s-}=k^{\prime}\}}d\tilde{N}_s^{k^{\prime}k}\right)\\
&=\prod_{k,k'\in
I}\mathcal{E}_t\left(\int_0^{\cdot}\eta^{k'k}_s(dN_s^{k'k}-q^{k'k}ds)\right)\\
&=\prod_{k,k'\in I}e^{-\int_0^t\eta_s^{k'k}q^{k'k}ds}\prod_{0<s\leq
t}(1+\eta^{k'k}_s\Delta N_s^{k'k}).
\end{align*}
The estimate (\ref{estimate_y_2}) in Theorem~\ref{theorem_ergodic_BSDE_system} implies that the difference of any two
components $\mathcal{Y}^{k}$ and $\mathcal{Y}^{k'}$ is bounded:
\begin{equation}\label{difference_esitmate}
|\mathcal{Y}_s^{k}-\mathcal{Y}_s^{k'}|\leq
\frac{1}{q^{\min}}\left(K_f+\frac{C_vC_{\eta}C_z}{(C_{\eta}-C_v)^2}\right),
\end{equation}
so $\eta^{k'k}$ is bounded. Since $1+\eta^{k'k}_s\Delta
N_s^{k'k}\geq 0,$ it follows that
$\mathcal{E}\left(\int_0^{\cdot}\sum_{k,k'\in
I}\eta^{k'k}_sd\tilde{N}_s^{k'k}\right)$ is a nonnegative bounded
$\mathbb{G}$-martingale. In turn, $\frac{(X_t(\pi))^{\delta}}{\delta}e^{\bar{\mathcal{Y}}_t^{\alpha_t}}$, $t\geq 0$, is a nonnegative local super-martingale, so it is a super-martingale for any $\pi\in\mathcal{A}^{\mathbb{G}}$, and
the super-martingale condition (\ref{supermartingale}) has been verified.

Finally, note that with
$\pi^{*}_s=\text{Proj}_{\Pi^{\alpha_{s-}}}\left(\frac{\mathcal{Z}_s^{\alpha_{s-}}+\theta^{\alpha_{s-}}(V_s)}{1-\delta}\right),$
we have
$$f^{\alpha_{s-}}(V_s,\mathcal{Z}_s^{\alpha_{s-}};\pi_s^*)-f^{\alpha_{s-}}(V_s,\mathcal{Z}_s^{\alpha_{s-}})=0.$$
The estimate (\ref{estimate_z_1}) in Theorem \ref{theorem_ergodic_BSDE_system} implies that $\mathcal{Z}^i$ is bounded, so the optimal trading strategy $\pi^{*}$ is also bounded and therefore $\pi^{*}\in\mathcal{A}^{\mathbb{G}}$. Note that
$\int_0^{\cdot}(\delta \pi_s^{i,*}+Z_s^{i})^{tr}dW_s$ is an
$\mathbb{F}$-BMO martingale. In turn,
$\mathcal{E}\left(\int_0^{\cdot}(\delta
\pi^*_s+Z_s^{\alpha_{s-}})^{tr}dW_s\right)$ is a uniformly
integrable $\mathbb{G}$-martingale. On the other hand, we have shown that $\mathcal{E}\left(\int_0^{\cdot}\sum_{k,k'\in
I}\eta^{k'k}_sd\tilde{N}_s^{k'k}\right)$ is a nonnegative bounded
$\mathbb{G}$-martingale. Thus, we easily conclude from (\ref{stochastic_exponential}) the martingale condition (\ref{martingale}) for $\frac{(X_t(\pi^*))^{\delta}}{\delta}e^{\bar{\mathcal{Y}}_t^{\alpha_t}}$, $t\geq 0$.


\subsection{Connection with classical utility maximization}

We provide an interpretation of the constant $\lambda ,$ appearing
in the
representation of the Markovian forward performance process (\ref{representation_2}%
), as the solution of the risk-sensitive control problem (\ref%
{ErgodicControlProblem}) below. It turns out that the constant
$\lambda $ is also the optimal long-term growth rate of the  utility
maximization problem (see (\ref%
{ErgodicControlProblem1}) below). For this, we need to shrink the admissible set $\mathcal{A}^{\mathbb{G}}$ to $\bar{\mathcal{A}}^{\mathbb{G}}$ defined as per below:
$$\bar{\mathcal{A}}^{\mathbb{G}}_{[0,t]}=\left\{\pi\in\mathcal{A}^{\mathbb{G}}_{[0,t]}:\
\int_0^{\cdot}(\pi_s^j)^{tr}dW_s\ \text{is an}\
\mathbb{F}\text{-BMO martingale}.\right\}$$
Let $\bar{\mathcal{A}}^{\mathbb{G}}=\cup _{t\geq
0}\bar{\mathcal{A}}^{\mathbb{G}}_{[0,t]}$. Note that for $\pi^*$ given in (\ref{optimal_strategy}), since it is bounded, we also have $\pi^*\in\bar{\mathcal{A}}^{\mathbb{G}}\subset\mathcal{A}^{\mathbb{G}}$.


\begin{proposition}
\label{propositionLambda} Let $T>0$ and ${\pi}\in
\bar{\mathcal{A}}^{\mathbb{G}}$. Define the probability measure
${\mathbb{P}}^{{\pi }}$ as
\begin{equation*}
\frac{d{\mathbb{P}}^{{\pi }}}{d\mathbb{P}}:=\mathcal{E}_T\left( \int_{0}^{\cdot}\delta {\pi }%
_{u}^{tr}dW_{u}\right),
\end{equation*}%
and the cost functional
\begin{equation*}
L^i(v;\pi):=\frac{1}{2}\delta (\delta-1)|{\pi }|^{2}+\delta\pi^{tr}
\theta^i (v),
\end{equation*}
for $i\in I$ and $(v,z)\in\mathbb{R}^d\times\mathbb{R}^d$.

Let $\left((\mathcal{Y}^{i},\mathcal{Z}^{i})_{i\in
I},\lambda\right)$ be the unique Markovian solution of the ergodic BSDE system
(\ref{ergodic_BSDE_system}) with driver $f^i$ as in
(\ref{driver_BSDE}), and satisfy (\ref{estmate_y_1}),
(\ref{estimate_z_1}) and (\ref{estimate_y_2}). Then, $\lambda $ is the long-term growth rate
of the risk-sensitive control problem
\begin{equation}
\lambda =\sup_{{\pi }\in \bar{\mathcal{A}}^{\mathbb{G}}}\limsup_{T\uparrow \infty }\frac{1}{T}%
\ln \mathbb{E}^{{\mathbb{P}}^{{\pi }}}\left[ e^{\int_{0}^{T}L^{\alpha_{s-}}(V_{s},{\pi }%
_{s})ds}\right],  \label{ErgodicControlProblem}
\end{equation}%
or, alternatively,
\begin{equation}
\lambda =\sup_{\pi \in \bar{\mathcal{A}}^{\mathbb{G}}}\limsup_{T\uparrow
\infty }\frac{1}{T}\ln \mathbb{E}\left[\frac{(X_{T}(\pi))^{\delta
}}{\delta }\right]. \label{ErgodicControlProblem1}
\end{equation}%
For both problems (\ref{ErgodicControlProblem})\ and (\ref%
{ErgodicControlProblem1}), the associated optimal control in each
regime $i$ is $\pi^{i,*}$ as in (\ref{optimal_strategy}).
\end{proposition}

\begin{proof}
We first observe that the driver $f^i$ in (\ref{driver_BSDE}) can be
written as
\begin{equation*}
f^i(v,z)=\sup_{{\pi}\in \Pi }\left( L^i(v,\pi)+z^{tr}\delta
{\pi}\right) +\frac{1}{2}|z|^{2}.
\end{equation*}%
Therefore, for arbitrary admissible $\tilde{\pi}$, we apply It\^o's
formula to the ergodic BSDE system (\ref{ergodic_BSDE_system}) on
$[T_{j-1},T_j)$, and obtain
\begin{align*}
&e^{\mathcal{Y}^{\alpha^{j-1}}_{T_{j}-}}-e^{\mathcal{Y}^{\alpha^{j-1}}_{T_{j-1}}}\\
=&\int_{T_{j-1}}^{T_{j}}e^{\mathcal{Y}_s^{\alpha^{j-1}}}\left[-\sup_{\pi_s^{\alpha^{j-1}}\in\Pi}
\left(L^{\alpha^{j-1}}(V_s,\pi^{\alpha^{j-1}}_s)+(\mathcal{Z}_s^{\alpha^{j-1}})^{tr}\delta\pi_s^{\alpha^{j-1}}\right)+(\mathcal{Z}_s^{\alpha^{j-1}})^{tr}\delta\tilde{\pi}_s^{\alpha^{j-1}}\right]ds\notag\\
&+\int_{T_{j-1}}^{T_{j}}e^{\mathcal{Y}_s^{\alpha^{j-1}}}\left[\lambda-\sum_{k\in
I}q^{\alpha^{j-1}k}(e^{\mathcal{Y}_s^{k}-\mathcal{Y}_s^{\alpha^{j-1}}}-1)\right]ds\notag\\
&+\int_{T_{j-1}}^{T_{j}}e^{\mathcal{Y}_s^{\alpha^{j-1}}}(\mathcal{Z}_s^{\alpha^{j-1}})^{tr}(dW_s-\delta\tilde{\pi}_s^{\alpha^{j-1}}ds).\notag
\end{align*}
In general, we decompose $e^{\mathcal{Y}_T^{\alpha_T}}$ into
continuous and jump components as
\begin{align*}
e^{\mathcal{Y}_T^{\alpha_T}}-e^{\mathcal{Y}_0^{i}}=&\ \sum_{j\geq
1}\left[e^{\mathcal{Y}_{T\wedge T_{j}-}^{\alpha^{j-1}}}-
e^{\mathcal{Y}_{T\wedge T_{j-1}}^{\alpha^{j-1}}}\right]+\sum_{j\geq
1}\left[e^{\mathcal{Y}_{T_{j}}^{\alpha^{j}}}- e^{\mathcal{Y}_{
T_{j}-}^{\alpha^{j-1}}}\right]\chi_{\{T_{j}\leq
T\}}\\
=&\ (I)+(II).
\end{align*}

It follows from the facts that $\alpha_{s-}=\alpha^{j-1}$,
$\pi_s=\pi_s^{\alpha^{j-1}}$ and
$\tilde{\pi}_s=\tilde{\pi}_s^{\alpha^{j-1}}$ for $s\in(T\wedge
T_{j-1},T\wedge T_j]$ that (I) has the expression
\begin{align}\label{local_Ito_2}
(I)=&\int_0^Te^{\mathcal{Y}_s^{\alpha_{s-}}}\left[-\sup_{\pi_s\in\Pi}
\left(L^{\alpha_{s-}}(V_s,\pi_s)+(\mathcal{Z}_s^{\alpha_{s-}})^{tr}\delta\pi_s\right)+(\mathcal{Z}_s^{\alpha_{s-}})^{tr}\delta\tilde{\pi}_s+\lambda\right]ds\notag\\
&-\int_{0}^{T}e^{\mathcal{Y}_s^{\alpha_{s-}}}\sum_{k\in
I}q^{\alpha_{s-}k}(e^{\mathcal{Y}_s^{k}-\mathcal{Y}_s^{\alpha_{s-}}}-1)ds\notag\\
&+\int_{0}^{T}e^{\mathcal{Y}_s^{\alpha_{s-}}}(\mathcal{Z}_s^{\alpha_{s-}})^{tr}(dW_s-\delta\tilde{\pi}_sds).
\end{align}
Furthermore, it follows from Lemma \ref{lemma_Ito} that (II) has the
expression
\begin{align}\label{local_Ito_3}
(II)=&\int_0^Te^{\mathcal{Y}_s^{\alpha_{s-}}}\sum_{k,k^{\prime}\in
I}\left(e^{\mathcal{Y}_s^k-\mathcal{Y}_s^{k^{\prime}}}-1\right)\chi_{\{\alpha_{s-}=k^{\prime}\}}d\tilde{N}_s^{k^{\prime}k}\notag\\
&+\int_0^Te^{\mathcal{Y}_s^{\alpha_{s-}}} \sum_{k\in
I}q^{\alpha_{s-}k}\left(e^{\mathcal{Y}_s^k-\mathcal{Y}_s^{\alpha_{s-}}}-1\right)ds.
\end{align}
Consequently, combining (\ref{local_Ito_2}) and (\ref{local_Ito_3}),
we obtain
\begin{align*}
e^{\mathcal{Y}_T^{\alpha_T}}-e^{\mathcal{Y}_0^{i}}=&\int_0^Te^{Y_s^{\alpha_{s-}}}\left[-\sup_{\pi_s\in\Pi}
\left(L^{\alpha_{s-}}(V_s,\pi_s)+(\mathcal{Z}_s^{\alpha_{s-}})^{tr}\delta\pi_s\right)+(\mathcal{Z}_s^{\alpha_{s-}})^{tr}\delta\tilde{\pi}_s+\lambda\right]ds\notag\\
&+\int_{0}^{T}e^{\mathcal{Y}_s^{\alpha_{s-}}}(\mathcal{Z}_s^{\alpha_{s-}})^{tr}dW_s^{{\mathbb{P}}^{\tilde{\pi}}}\notag\\
&+\int_0^Te^{\mathcal{Y}_s^{\alpha_{s-}}}\sum_{k,k^{\prime}\in
I}\left(e^{\mathcal{Y}_s^k-\mathcal{Y}_s^{k^{\prime}}}-1\right)\chi_{\{\alpha_{s-}=k^{\prime}\}}d\tilde{N}_s^{k^{\prime}k},
\end{align*}
where the process $W_{t}^{{\mathbb{P}}^{\tilde{\pi}}}:=W_{t}-\int_{0}^{t}%
\delta \tilde{\pi}_{u}du$, $t\geq 0$, is a Brownian motion under ${\mathbb{P}%
}^{\tilde{\pi}}$. In turn,%
\begin{align*}
e^{\mathcal{Y}_T^{\alpha_T}}=&\ e^{\mathcal{Y}_0^{i}+\lambda
T}\mathcal{E}_T\left( \int_{0}^{\cdot
}(\mathcal{Z}_{s}^{\alpha_{s-}})^{tr}dW_{s}^{\mathbb{P}^{\tilde{\pi}}}\right)
\mathcal{E}_T\left(\sum_{k,k^{\prime}\in
I}\left(e^{\mathcal{Y}_s^k-\mathcal{Y}_s^{k^{\prime}}}-1\right)\chi_{\{\alpha_{s-}=k^{\prime}\}}d\tilde{N}_s^{k^{\prime}k}\right)
\\
&\ \times
e^{-\int_{0}^{T}L^{\alpha_{s-}}(V_{s},\tilde{\pi}_{s})ds}\\
&\  \times e^{\int_{0}^{T}\left[\left(L^{\alpha_{s-}}(V_{s},\tilde{%
\pi}_{s})+(\mathcal{Z}_{s}^{\alpha_{s-}})^{tr}\delta
\tilde{\pi}_{s}\right)-\sup_{{\pi }_{s}\in \Pi }\left(
L^{\alpha_{s-}}(V_{s},\pi
_{s})+(\mathcal{Z}_{s}^{\alpha_{s-}})^{tr}\delta {\pi }_{s}\right)
\right]ds}.
\end{align*}%

Next, we observe that for any $\tilde{\pi}\in\bar{\mathcal{A}}^{\mathbb{G}}$, the last exponential term on the right
hand side is bounded above by $1$. Taking expectation under
$\mathbb{P}^{\tilde{\pi}}$ then yields
\begin{align*}
&\mathbb{E}^{{\mathbb{P}}^{\tilde{\pi}}}\left[ e^{\int_{0}^{T}L^{\alpha_{s-}}(V_{s},%
\tilde{\pi}_{s})ds}\right]e^{-\mathcal{Y}_0^i-\lambda
T}\\
\leq &\
\mathbb{E}^{{\mathbb{P}}^{\tilde{\pi}}}\left[e^{-\mathcal{Y}_T^{\alpha_T}}
\mathcal{E}_T\left( \int_{0}^{\cdot
}(\mathcal{Z}_{s}^{\alpha_{s-}})^{tr}dW_{s}^{\mathbb{P}^{\tilde{\pi}}}\right)
\mathcal{E}_T\left(\sum_{k,k^{\prime}\in
I}\left(e^{\mathcal{Y}_s^k-\mathcal{Y}_s^{k^{\prime}}}-1\right)\chi_{\{\alpha_{s-}=k^{\prime}\}}d\tilde{N}_s^{k^{\prime}k}\right)
\right].
\end{align*}%
Define the probability measure $\mathbb{Q}^{\tilde{\pi}}$ as
\begin{equation*}
\frac{d\mathbb{Q}^{\tilde{\pi}}}{d{\mathbb{P}}^{\tilde{\pi
}}}:=\mathcal{E}_T\left( \int_{0}^{\cdot
}(\mathcal{Z}_{s}^{\alpha_{s-}})^{tr}dW_{s}^{\mathbb{P}^{\tilde{\pi}}}\right)
\mathcal{E}_T\left(\int_0^{\cdot}\sum_{k,k^{\prime}\in
I}\left(e^{\mathcal{Y}_s^k-\mathcal{Y}_s^{k^{\prime}}}-1\right)\chi_{\{\alpha_{s-}=k^{\prime}\}}d\tilde{N}_s^{k^{\prime}k}\right).
\end{equation*}%
Then, it follows from the linear growth condition (\ref{estmate_y_1}) of
$\mathcal{Y}_T^i=\mathbf{y}^i(V_T)$ and Assumption
\ref{assumption_13} on $V$ that
\begin{equation*}
\frac{1}{C}\leq \mathbb{E}^{\mathbb{Q}^{\tilde{\pi}}}\left(
e^{-\mathcal{Y}_{T}^{\alpha_T}}\right) \leq C,
\end{equation*}%
for some constant $C$ independent of $T$ (see
(\ref{first_estimate})). Consequently,
\begin{equation*}
\frac{1}{T}\ln \mathbb{E}^{{\mathbb{P}}^{\tilde{\pi}}}\left[
e^{\int_{0}^{T}L^{\alpha_{s-}}(V_{s},\tilde{\pi}_{s})ds}\right]\leq
\lambda +\frac{Y_{0}^i}{T}+\frac{1}{T}\ln
\mathbb{E}^{\mathbb{Q}^{{\tilde{\pi}}}}\left(
e^{-\mathcal{Y}^{\alpha_T}_{T}}\right).
\end{equation*}%
Sending $T\rightarrow\infty$, we obtain, for any
$\tilde{\pi}\in\bar{\mathcal{A}}^{\mathbb{G}}$,
$$\lambda\geq \limsup_{T\uparrow \infty }\frac{1}{T}%
\ln \mathbb{E}^{{\mathbb{P}}^{{\pi }}}\left[ e^{\int_{0}^{T}L^{\alpha_{s-}}(V_{s},\tilde{\pi }%
_{s})ds}\right],$$ with equality choosing $\tilde{\pi}_{s}=\pi
_{s}^{\ast }$, with $\pi _{s}^{\ast }$ as in
(\ref{optimal_strategy}).

To show that $\lambda $ also solves (\ref{ErgodicControlProblem1}),
we observe that for $\pi \in \bar{\mathcal{A}}^{\mathbb{G}}$, we have
\begin{align*}
\mathbb{E}\left[ \frac{(X_{T}^{\pi })^{\delta }}{\delta }\right]&=\frac{%
X_{0}^{\delta }}{\delta}\mathbb{E}\left[
e^{\int_{0}^{T}L^{\alpha_{s-}}(V_s,\pi
_{s})ds}\mathcal{E}_T\left(\int_{0}^{\cdot }\delta {\pi }_{s}^{tr}dW_{s}%
\right) _{T}\right]\\
&=\frac{x^{\delta }}{\delta }\mathbb{E}^{\mathbb{P}^{\pi }}\left[
e^{\int_{0}^{T}L^{\alpha_{s-}}(V_{s},\pi _{s})ds}\right],
\end{align*}%
and the rest of the arguments follow.
\end{proof}


\section{Application to the large time behavior of PDE systems with quadratic growth Hamiltonians}\label{section: large_time}

As the second application, we use the ergodic BSDE system
(\ref{ergodic_BSDE_system}) to study the large time behavior of the
PDE system with quadratic growth Hamiltonians, namely
\begin{align}\label{PDE_system}
-\partial_t\mathbf{y}^i(t,v)&+\frac12\text{Trace}(\kappa^{tr}\kappa\nabla_v^2\mathbf{y}^i(t,v))+\eta(v)^{tr}\nabla_v\mathbf{y}^i(t,v)\notag\\
&+f^i(v,\kappa^{tr}\nabla_v\mathbf{y}^i(t,v))+\sum_{k\in
I}q^{ik}\left(e^{(\mathbf{y}^k-\mathbf{y}^i)(t,v)}-1\right)=0,
\end{align}
with initial condition $\mathbf{y}^i(0,v)=h^i(v)$, for
$(t,v)\in\mathbb{R}_+\times\mathbb{R}^d$ and $i\in I$. The data
$\kappa,\eta(\cdot),f^i(\cdot,\cdot)$ and $q^{ik}$ of the PDE system
are assumed to satisfy Assumptions
\ref{assumption_11}-\ref{assumption_15} and, moreover, the initial
condition $h^i(\cdot)$ is bounded and Lipschitz continuous. Due to
Assumption \ref{assumption_11}(ii), the Hamiltonians
$f^i(\cdot,\cdot)$ has quadratic growth in the gradients $\nabla_v
\mathbf{y}^i(t,v)$. For this reason, (\ref{PDE_system}) is dubbed as
\emph{a PDE system with quadratic growth Hamiltonians}. A special
case of the above PDE system (\ref{PDE_system}) has been considered
in \cite{Bec0} and \cite{Bec} to study the utility indifference
prices of financial derivatives in a regime switching market.

The scalar case of (\ref{PDE_system}) and its large time behavior has been studied in \cite{Hu11} using the ergodic BSDE approach. We extend their result from the scalar case to the system of equations. First, we provide a probabilistic representation for the PDE system (\ref{PDE_system}).
For $T>0$, let $(\mathcal{Y}^{i,v}(T),\mathcal{Z}^{i,v}(T))_{i\in I}$ be a solution to
the finite horizon BSDE system
\begin{align}\label{fhorizon_BSDE_system_truncation_111}
\mathcal{Y}_t^{i,v}(T)=&\
h^{i}(V_T^{v})+\int_t^T\left[f^i(V_s^v,\mathcal{Z}_s^{i,v}(T))+\sum_{k\in
I}q^{ik}(e^{\mathcal{Y}_s^{k,v}(T)-\mathcal{Y}_s^{i,v}(T)}-1)\right]ds\notag\\
&-\int_t^T(\mathcal{Z}_s^{i,v}(T))^{tr}dW_s.
\end{align}
Following along the similar arguments used to solve the finite
horizon BSDE system (\ref{fhorizon_BSDE_system_truncation}) (see
section \ref{subsection_estimate} with $\rho=0$), we deduce that
$(\mathcal{Y}^{i,v}(T),\mathcal{Z}^{i,v}(T))_{i\in I}$ is actually
the unique bounded solution of
(\ref{fhorizon_BSDE_system_truncation_111}) with
\begin{equation}\label{bound_of_Z_11}
|\mathcal{Z}_t^{i,v}(T)|\leq\frac{C_v}{C_{\eta}-C_{v}}+C_h.
\end{equation}
Note that the bound of $\mathcal{Y}^{i,v}(T)$ may depend on $T$.  Furthermore, following from \cite[Theorems 3.4 and 3.5]{Barles}, we deduce that $\mathbf{y}^i(\cdot,\cdot)$, defined as $\mathbf{y}^i(T-t,V_t^v):=\mathcal{Y}_t^{i,v}(T)$, is the unique viscosity solution to the PDE system (\ref{PDE_system}). Since the monotone condition for $\mathbf{y}^k$ in the last nonlinear term of (\ref{PDE_system}) holds, a comparison result similar to Lemma \ref{comparsion_lemma} also holds for (\ref{PDE_system}) (see Remark 3.9 in \cite{Barles}).


\begin{theorem}\label{theorem_large_time_behavior} Suppose that Assumptions
\ref{assumption_11}-\ref{assumption_15} hold, and $h^i(\cdot)$,
$i\in I$, is bounded by a constant $K_h$ and Lipschitz continuous
with its Lipschitz constant $C_h$.

Let $\left((\mathcal{Y}^{i,v},\mathcal{Z}^{i,v})_{i\in
I},\lambda\right)$ be the unique Markovian solution of the ergodic BSDE system
(\ref{ergodic_BSDE_system}) with
$\mathcal{Y}_t^{i,v}=\mathbf{y}^i(V_t^v)$ and
$\mathcal{Z}_t^{i,v}=\mathbf{z}^i(V_t^v)$ satisfying
(\ref{estmate_y_1}), (\ref{estimate_z_1}) and (\ref{estimate_y_2}). Let $\mathbf{y}^i(\cdot,\cdot)$ be the unique viscosity solution to the PDE system (\ref{PDE_system}).
Then, there exists a constant $L$, independent of $v\in\mathbb{R}^d$
and $i\in I$, such that
\begin{equation}\label{first_asymptotic}
\lim_{T\rightarrow \infty}(\mathbf{y}^i(T,v)-\lambda
T-\mathbf{y}^i(v))=L,
\end{equation}
and moreover, there exist constants $C$ and $K_v$, independent of
$T$, such that
\begin{equation}\label{second_asymptotic}
|\mathbf{y}^i(T,v)-\lambda T-\mathbf{y}^i(v)-L|\leq
C(1+|v|^2)e^{-K_vT}.
\end{equation}
\end{theorem}

\begin{proof} The proof is adapted from the arguments in
\cite[Section 4.2]{Hu11}  (see also \cite{Hu12}). In the following, we only highlight the key
difference from their proof.

We first convert the BSDE system
(\ref{fhorizon_BSDE_system_truncation_111}) to a scalar-valued BSDE
driven by the Brownian motion $W$ and the Markov chain $\alpha$. To
this end, similar to Appendix \ref{Appendix_B}, for $t\in[0,T]$ and
$v\in\mathbb{R}^d$, we introduce
$$\mathcal{Y}_t^{v}(T):=\mathcal{Y}_t^{\alpha_t,v}(T)=\mathbf{y}^{\alpha_t}(T-t,V_t^v),$$
$$\mathcal{Z}_t^{v}(T):=\mathcal{Z}_t^{\alpha_{t-},v}(T)=\mathbf{z}^{\alpha_{t-}}(T-t,V_t^v),$$
and for $k',k\in I$,
$$\mathcal{U}_t^{v}(k',k;T):=\mathcal{Y}_t^{k,v}(T)-\mathcal{Y}_t^{k',v}(T)=(\mathbf{y}^{k}-\mathbf{y}^{k'})(T-t,V_t^v).$$
Then, following along the similar arguments in the proof of Lemma \ref{lemma_difference}, we deduce that
\begin{equation}\label{estimate_U_T}
|\mathcal{U}_t^{v}(k',k;T)|\leq
\frac{1}{q^{\min}}\left(K_f+\frac{C_vC_{\eta}C_z}{(C_{\eta}-C_v)^2}\right)+K_h.
\end{equation}
In turn, using Lemma \ref{lemma_Ito}, we deduce that
$(\mathcal{Y}^v(T),\mathcal{Z}^{v}(T),(\mathcal{U}^{v}(k',k;T))_{k',k\in
I})$ satisfies the scalar-valued BSDE driven by $W$ and $\alpha$,
i.e. for $t\in[0,T]$,
\begin{align}\label{ergodic_BSDE_jumps_11}
\mathcal{Y}_t^{v}(T)=&\
h^{\alpha_T}(V_T^{v})+\int_t^Tf^{\alpha_{s-}}(V_s^{v},\mathcal{Z}_s^{v}(T))dt-\int_t^T(\mathcal{Z}_s^{v}(T))^{tr}dW_s\notag\\
&+\int_t^T\sum_{k\in
I}q^{\alpha_{s-}k}\left[e^{\mathcal{U}_s^{v}(\alpha_{s-},k;T)}-1-\mathcal{U}_s^{v}(\alpha_{s-},k;T)\right]ds\notag\\
&-\int_t^T\sum_{k,k'\in
I}\mathcal{U}_s^{v}(k',k;T)\chi_{\{\alpha_{s-}=k'\}}d\tilde{N}_s^{k'k}.
\end{align}

Next, we define $\delta
\mathcal{Y}_t^{v}(T):=\mathbf{y}^{\alpha_t}(T-t,V_t^v)-\mathbf{y}^{\alpha_t}(V_t^v)-\lambda
(T-t)$  for $t\in[0,T]$. Then, we have the following key estimates.

\begin{lemma}\label{lemma_estimate}
The function
$\delta\mathcal{Y}_0^{v}(T)=\mathbf{y}^{i}(T,v)-\mathbf{y}^{i}(v)-\lambda
T, v\in \mathbb{R}^d,$ admits the following properties: There exist constants $C$ and
$K_v$, independent of $T$, such that for  arbitrary $v_1, v_2\in \mathbb{R}^d, $

(i) $|\delta\mathcal{Y}_0^{v_1}(T)|\leq C(1+|v_1|)$;

(ii) $|\delta\mathcal{Y}_0^{v_1}(T)-\delta\mathcal{Y}_0^{v_2}(T)|\leq C|v_1-v_2|$;

 (iii) $\left|\delta\mathcal{Y}_0^{v_1}(T)-\delta\mathcal{Y}_0^{v_2}(T)\right|\leq C(1+|v_1|^2+|v_2|^2)e^{-K_vT}.$

\end{lemma}

\begin{proof} First, we prove Assertion (ii). When $\eta$ and $f$ are continuously differentiable functions with bounded derivatives, noting  $$\kappa^{tr}\nabla_v\mathbf{y}^i(T-t,V_t^v)=\mathcal{Z}_t^{i,v}(T) \quad \mbox{\rm
and} \quad \kappa^{tr}\nabla_v\mathbf{y}^i(V_t^v)=\mathcal{Z}_t^{i,v}, $$
the desired  assertion follows from the boundedness of both
$\mathcal{Z}_t^{i,v}(T)$ and $\mathcal{Z}_t^{i,v}$ (cf.
(\ref{bound_of_Z_11}) and (\ref{estimate_z_1})) and Assumption
\ref{assumption_13} on $\kappa$. For our general $\eta$ and $f$, Assertion (ii) can be proved by a standard mollification argument.

Next, we prove the assertions (i) and (iii). To this end, for
$t\in[0,T]$, define
$$\delta
\mathcal{Z}_t^{v}(T):={\mathcal{Z}}^v_t(T)-{\mathcal{Z}}^v_t, \quad \mbox{\rm and } \delta
\mathcal{U}_t^{v}(k',k;T):={\mathcal{U}}^v_t(k',k;T)-{\mathcal{U}}^v_t(k',k)$$
with $k',k\in I$.
Then, we deduce from (\ref{ergodic_BSDE_jumps_11}) and
(\ref{ergodic_BSDE_jumps})  that
$(\delta\mathcal{Y}^{v}(T),\delta\mathcal{Z}^v(T),(\delta\mathcal{U}^{v}(k',k;T))_{k',k\in
I})$ satisfies
\begin{align}\label{difference_ergodic_BSDE_jumps_11}
\delta\mathcal{Y}_0^{v}(T)=&\
h^{\alpha_T}(V_T^v)-\mathbf{y}^{\alpha_T}(V_T^{v})\notag\\
&+\int_0^T\left[f^{\alpha_{s-}}(V_s^{v},\mathcal{Z}_s^{v}(T))-f^{\alpha_{s-}}(V_s^{v},{\mathcal{Z}}_s^{v})\right]ds-\int_0^T(\delta\mathcal{Z}_s^{v}(T))^{tr}dW_s\notag\\
&+\int_0^T\sum_{k\in
I}q^{\alpha_{s-}k}\left[g(\mathcal{U}_s^{v}(\alpha_{s-},k;T))-g({\mathcal{U}}_s^{v}(\alpha_{s-},k))\right]ds\notag\\
&-\int_0^T\sum_{k,k'}\delta\mathcal{U}_s^{v}(k',k;T)\chi_{\{\alpha_{s-}=k'\}}d\tilde{N}_s^{k'k},
\end{align}
where $g(\cdot)$ is given in (\ref{def_g}). Since
$\mathcal{Z}^v(T)$, $\mathcal{Z}^v$, $\mathcal{U}^v(k',k;T)$ and
$\mathcal{U}^v(k',k)$ are all uniformly bounded (cf.
(\ref{bound_of_Z_11}), (\ref{estimate_z_1}), (\ref{estimate_U_T})
and (\ref{estimate_y_2})), analogous to Appendix \ref{Appendix_B},
we may introduce an equivalent probability measure $\mathbb{Q}$,
under which we have
\begin{equation}\label{change_of_measure}
\delta\mathcal{Y}_0^{v}(T)=\mathbb{E}^{\mathbb{Q}}[h^{\alpha_T}(V_T^v)-\mathbf{y}^{\alpha_T}(V_T^{v})].
\end{equation}
Since both $h^{i}(\cdot)$ and $\mathbf{y}^{i}(\cdot)$ with $i\in I$,
have at most a linear growth, we deduce assertion (i) from the
estimate in (\ref{first_estimate}).

To prove assertion (iii), from (\ref{change_of_measure}), we have,
for $v,\bar{v}\in\mathbb{R}^d$,
$$\delta\mathcal{Y}_0^{v}(T)-\delta\mathcal{Y}_0^{\bar{v}}(T)=
\mathbb{E}^{\mathbb{Q}}\left[(h^{\alpha_T}(V_T^v)-\mathbf{y}^{\alpha_T}(V_T^{v}))-(h^{\alpha_T}(V_T^{\bar{v}})-\mathbf{y}^{\alpha_T}(V_T^{\bar{v}}))\right].$$
The conclusion then follows from the linear growth of both $h^{i}(\cdot)$
and $\mathbf{y}^{i}(\cdot)$ for $i\in I$, and the estimate in (\ref{second_estimate}).
\end{proof}

Let us return to the proof of Theorem \ref{theorem_large_time_behavior}. Using the first estimate (i) in Lemma \ref{lemma_estimate}, by a
standard diagonal procedure, we may construct a sequence $\{T_k\}$
such that
$$\lim_{T_k\rightarrow\infty}(\mathbf{y}^i(T_k,v)-\mathbf{y}^i(v)-\lambda
T_k)=L(v)$$
 for some limit function $L(v)$. Moreover, the second
estimate (ii) in Lemma \ref{lemma_estimate} implies that the limit
function $L(v)$ can be extended to a Lipschitz continuous function,
and the third estimate (iii) in Lemma \ref{lemma_estimate} further
implies that the limit actually satisfies $L(v)=L$ with $L$ being a
constant. This establishes the limit
(\ref{first_asymptotic}).

To show the convergence rate (\ref{second_asymptotic}), we deduce from
(\ref{first_asymptotic}) and (\ref{change_of_measure})
that, for $T'>T$,
\begin{align*}
|\delta\mathcal{Y}_0^{v}(T)-L|&=
\lim_{T'\rightarrow\infty}|\delta\mathcal{Y}_0^{v}(T)-\delta\mathcal{Y}_0^{v}(T')|\\
&=\lim_{T'\rightarrow\infty}\left|\delta\mathcal{Y}_0^{v}(T)-\mathbb{E}^{\mathbb{Q}}\left[h^{\alpha^{m(T')}_{T'}}(V_{T'}^v)-\mathbf{y}^{\alpha^{m(T')}_{T'}}(V_{T'}^{v})\right]\right|,
\end{align*}
where $m(T'):=2i-\alpha_{T'-T}^i$. Here we use $\alpha^i$ to
emphasize the initial data of the Markov chain $\alpha_0=i$. It then
follows from the tower property of conditional expectations that,
\begin{align*}
\mathbb{E}^{\mathbb{Q}}\left[h^{\alpha^{m(T')}_{T'}}(V_{T'}^v)-\mathbf{y}^{\alpha^{m(T')}_{T'}}(V_{T'}^{v})\right]=&\
\mathbb{E}^{\mathbb{Q}}\left[\mathbb{E}^{\mathbb{Q}}\left[h^{\alpha^{m(T')}_{T'}}(V_{T'}^v)-\mathbf{y}^{\alpha_{T'}^{m(T')}}(V_{T'}^{v})|\mathcal{G}_{T'-T}\right]\right]\\
=&\
\mathbb{E}^{\mathbb{Q}}\left[\mathbf{y}^{\alpha^{m(T')}_{T'-T}}(T,V_{T'-T}^v)-\mathbf{y}^{\alpha^{m(T')}_{T'-T}}(V_{T'-T}^v)-\lambda
T\right]\\
=&\
\mathbb{E}^{\mathbb{Q}}\left[\mathbf{y}^{i}(T,V_{T'-T}^v)-\mathbf{y}^{i}(V_{T'-T}^v)-\lambda
T\right]
\end{align*}
where we also used the relationship
$\alpha_{T'-T}^{m(T')}=\alpha_{T'-T}^{i-(\alpha_{T'-T}^i-i)}=i$ in
the last equality. In turn, using the definition  $\delta
\mathcal{Y}_0^{v}(T)=\mathbf{y}^{i}(T,v)-\mathbf{y}^{i}(v)-\lambda
T$, we obtain
\begin{align*}
|\delta\mathcal{Y}_0^{v}(T)-L|=&\lim_{T'\rightarrow\infty}\left|\delta\mathcal{Y}_0^{v}(T)-\mathbb{E}^{\mathbb{Q}}\left[h^{\alpha^{m(T')}_{T'}}(V_{T'}^v)-\mathbf{y}^{\alpha^{m(T')}_{T'}}(V_{T'}^{v})\right]\right|\\
=\
&\lim_{T'\rightarrow\infty}\mathbb{E}^{\mathbb{Q}}\left[\mathbf{y}^{i}(T,v)-\mathbf{y}^{i}(v)-(\mathbf{y}^{i}(T,V_{T'-T}^v)-\mathbf{y}^{i}(V_{T'-T}^v))\right],\\
\leq&
\lim_{T'\rightarrow\infty}C\left(1+|v|^2+\mathbb{E}^{\mathbb{Q}}\left[|V_{T'-T}^v|^2\right]\right)e^{-K_vT},
\end{align*}
where the assertion (iii) in Lemma \ref{lemma_estimate} 	is used in
the last inequality. The convergence rate then follows from the
moment estimate (\ref{first_estimate}).  The proof of Theorem \ref{theorem_large_time_behavior} is complete.
\end{proof}


%
%
%

\section{Conclusions}

In this paper, we introduced and solved a new type of quadratic BSDE
systems in an infinite time horizon and, subsequently, derived their
asymptotic limit as ergodic BSDE systems. The ergodic BSDE system is
used to characterize Markovian regime switching forward performance processes
and their associated optimal portfolio strategies. We have also
shown a connection between Markovian regime switching forward performance
processes and their classical expected utility counterparts via the
constant $\lambda$ in the corresponding ergodic BSDE system.
Finally, we use the ergodic BSDE system to study the large time
behavior for a class of PDE systems with quadratic growth
Hamiltonians.

\appendix
\section{Proof of Lemma \ref{comparsion_lemma}}\label{Appendix_A}
The idea of the proof is adapted from the arguments used in
\cite{HuPeng}. For $t\in[0,T]$, let
\begin{align*}
\delta Y_t^i:=Y_t^i-\bar{Y}_t^i,\ \ \delta Z_t^i:=Z_t^i-\bar{Z}_t^i\
\ \text{and}\ \ \delta \xi^i:=\xi^i-\bar{\xi}^i.
\end{align*}
Applying It\^o's formula to $(\delta Y_t^{i+})^2$ yields
\begin{align*}
(\delta Y_t^{i+})^2=&\ (\delta\xi^{i+})^2+\int_t^T2\delta
Y_s^{i+}[F_s^{i}(Z_s^i)-\bar{F}_s^{i}(\bar{Z}_s^i)]ds\\
&+\int_t^T2\delta
Y_s^{i+}[G_s^{i}(Y_s^i,Y_s^{-i})-\bar{G}_s^{i}(\bar{Y}_s^i,\bar{Y}_s^{-i})]ds\\
&-\int_t^T\chi_{\{\delta Y_s^i>0\}}|\delta
Z_s^i|^2ds-\int_t^T2\delta Y_s^{i+}(\delta Z_s^i)^{tr}dW_s.
\end{align*}
Using (\ref{Lip_F}) and (\ref{compare_F}), we obtain
\begin{align*}
F_s^{i}(Z_s^i)-\bar{F}_s^{i}(\bar{Z}_s^i)&=F_s^{i}(Z_s^i)-{F}_s^{i}(\bar{Z}_s^i)
+F_s^{i}(\bar{Z}_s^i)-\bar{F}_s^{i}(\bar{Z}_s^i)\leq C_f|\delta
Z_s^i|.
\end{align*}
Using (\ref{Lip_G}) and (\ref{compare_G}), together with the
monotone condition of $G^i_s$, we further obtain
\begin{align*}
&G_s^{i}(Y_s^i,Y_s^{-i})-\bar{G}_s^{i}(\bar{Y}_s^i,\bar{Y}_s^{-i})\\
=&\
G_s^{i}(Y_s^i,Y_s^{-i})-{G}_s^{i}(\bar{Y}_s^i,\bar{Y}_s^{-i})+G_s^{i}(\bar{Y}_s^i,\bar{Y}_s^{-i})-\bar{G}_s^{i}(\bar{Y}_s^i,\bar{Y}_s^{-i})\\
\leq&\ C_g\left(|\delta Y_s^i|+\sum_{k\neq i}\delta Y_s^{k+}\right).
\end{align*}
In turn, since $\delta\xi^{i+}=0$, we have
\begin{align*}
&\mathbb{E}[(\delta Y_t^{i+})^2]\\
\leq&\ \mathbb{E}\left[\int_t^T\left(2C_f\delta Y_s^{i+}|\delta
Z_s^i|+2C_g\delta Y_s^{i+}(|\delta Y_s^i|+\sum_{k\neq i}\delta
Y_s^{k+})-\chi_{\{\delta Y_s^i>0\}}|\delta Z_s^i|^2\right)ds\right]\\
\leq&\ \mathbb{E}\left[\int_t^T\chi_{\{\delta
Y_s^i>0\}}\left(-|\delta Z_s^i|^2+2C_f\delta Y_s^i|\delta
Z_s^i|-C_f^2(\delta
Y_s^i)^2\right)ds\right]\\
&\ +\mathbb{E}\left[\int_t^T\left((2C_g+C_f^2)(\delta
Y_s^{i+})^2+C_g^2(\delta Y_s^{i+})^2+\sum_{k\neq i}(\delta
Y_s^{k+})^2\right)ds\right].
\end{align*}
Thus, there exists a constant $C$ such that
$$\sum_{i\in I}\mathbb{E}[(\delta Y_t^{i+})^2]\leq C\int_t^T\sum_{i\in I}
\mathbb{E}[(\delta Y_s^{i+})^2]ds.$$ It then follows from Gronwall's
inequality that $\mathbb{E}[(\delta Y_t^i)^2]=0$, for $t\in[0,T]$
and $i\in I$, so $Y_t^{i}\leq \bar{Y}_t^i$ and we conclude.

\section{Proof of Theorem
\ref{theorem_ergodic_BSDE_system}}\label{Appendix_B} Let $\alpha$ be
the Markov chain introduced in section
\ref{section:regime_switching_forward} satisfying Assumptions
\ref{assumption_12} and \ref{assumption_15}. Let
$\left((\mathcal{Y}^{i,v},\mathcal{Z}^{i,v})_{i\in
I},\lambda\right)$ and
$\left((\bar{\mathcal{Y}}^{i,v},\bar{\mathcal{Z}}^{i,v})_{i\in
I},\bar{\lambda}\right)$ be two Markovian solutions to the ergodic BSDE system
(\ref{ergodic_BSDE_system}) both satisfying (\ref{estmate_y_1}),
(\ref{estimate_z_1}) and (\ref{estimate_y_2}).

For $t\geq 0$ and $v\in\mathbb{R}^d$, define
$$\mathcal{Y}_t^{v}:=\mathcal{Y}_t^{\alpha_t,v}=\mathbf{y}^{\alpha_t}(V_t^v),$$
$$\mathcal{Z}_t^{v}:=\mathcal{Z}_t^{\alpha_{t-},v}=\mathbf{z}^{\alpha_{t-}}(V_t^v),$$
and for $k',k\in I$,
$$\mathcal{U}_t^{v}(k',k):=\mathcal{Y}_t^{k,v}-\mathcal{Y}_t^{k',v}=(\mathbf{y}^{k}-\mathbf{y}^{k'})(V_t^v).$$
We may also define
$(\bar{\mathcal{Y}}^v,\bar{\mathcal{Z}}^v,(\bar{\mathcal{U}}^v(k',k))_{k',k\in
I})$ in an analogous way. Furthermore, let $\delta
\mathcal{Y}_t^{v}:={\mathcal{Y}}^v_t-\bar{\mathcal{Y}}^v_t$, $\delta
\mathcal{Z}_t^{v}:={\mathcal{Z}}^v_t-\bar{\mathcal{Z}}^v_t$, $\delta
\mathcal{U}_t^{v}(k',k):={\mathcal{U}}^v_t(k',k)-\bar{\mathcal{U}}^v_t(k',k)$
and $\delta\lambda:=\lambda-\bar{\lambda}$.

First, using Lemma \ref{lemma_Ito}, we deduce that
$(\mathcal{Y}^v,\mathcal{Z}^{v},(\mathcal{U}^{v}(k',k))_{k',k\in
I},\lambda)$ satisfies the scalar-valued ergodic BSDE driven by the
Brownian motion $W$ and the Markov chain $\alpha$, i.e. for $t\geq
0$,
\begin{align}\label{ergodic_BSDE_jumps}
d\mathcal{Y}_t^{v}=&\
-f^{\alpha_{t-}}(V_t^{v},\mathcal{Z}_t^{v})dt-\sum_{k\in
I}q^{\alpha_{t-}k}\left[e^{\mathcal{U}_t^{v}(\alpha_{t-},k)}-1-\mathcal{U}_t^{v}(\alpha_{t-},k)\right]dt+\lambda
dt\notag\\
&+(\mathcal{Z}_t^{v})^{tr}dW_t+\sum_{k,k'\in
I}\mathcal{U}_t^{v}(k',k)\chi_{\{\alpha_{t-}=k'\}}d\tilde{N}_t^{k'k}.
\end{align}
In turn,
$(\delta\mathcal{Y}^{v},\delta\mathcal{Z}^v,(\delta\mathcal{U}^{v}(k',k))_{k',k\in
I},\delta\lambda)$ satisfies
\begin{align}\label{difference_ergodic_BSDE_jumps}
d(\delta\mathcal{Y}_t^{v})=&\
-\left[f^{\alpha_{t-}}(V_t^{v},\mathcal{Z}_t^{v})-f^{\alpha_{t-}}(V_t^{v},\bar{\mathcal{Z}}_t^{v})\right]dt+(\delta\mathcal{Z}_t^{v})^{tr}dW_t\notag\\
&-\sum_{k\in
I}q^{\alpha_{t-}k}\left[g(\mathcal{U}_t^{v}(\alpha_{t-},k))-g(\bar{\mathcal{U}}_t^{v}(\alpha_{t-},k))\right]dt\notag\\
&+\sum_{k,k'}\delta\mathcal{U}_t^{v}(k',k)\chi_{\{\alpha_{t-}=k'\}}d\tilde{N}_t^{k'k}+\delta\lambda
dt,
\end{align}
where
\begin{equation}\label{def_g}
g(x):=e^{x}-1-x,\ \text{with}\ |x|\leq
\frac{1}{q^{\min}}\left(K_f+\frac{C_vC_{\eta}C_z}{(C_{\eta}-C_v)^2}\right).
\end{equation}

Next, we introduce
$$\delta f^{\alpha_{t-}}(V_t^v):=\frac{f^{\alpha_{t-}}(V_t^{v},\mathcal{Z}_t^{v})-f^{\alpha_{t-}}(V_t^{v},\bar{\mathcal{Z}}_t^{v})}{|\delta\mathcal{Z}_t^{v}|^2}\delta\mathcal{Z}_t^{v}\chi_{\{\delta\mathcal{Z}_t^{v}\neq 0\}},$$
and, for $k\in I$,
$$\delta g^{\alpha_{t-}k}(V_t^v):=\frac{g(\mathcal{U}_t^{v}(\alpha_{t-},k))-g(\bar{\mathcal{U}}_t^{v}(\alpha_{t-},k))}
{\delta\mathcal{U}_t^{v}(\alpha_{t-},k)}\chi_{\{\delta\mathcal{U}_t^{v}(\alpha_{t-},k)\neq
0\}}.$$ Note that Assumption \ref{assumption_11}(ii) and
(\ref{estimate_z_1}) imply that $\delta f^{\alpha_{t-}}(V_t^{v})$,
$t\geq 0$, is uniformly bounded. Moreover, the mean value theorem
(applied to the function $g(\cdot)$) and (\ref{estimate_y_2}) imply
that $\delta g^{\alpha_{t-}k}(V_t^v)$, $t\geq 0$, is also uniformly
bounded. Thus, for any $T>0$, define an equivalent probability
measure $\mathbb{Q}$ as
\begin{equation*}
\frac{d\mathbb{Q}}{d{\mathbb{P}}}:=\mathcal{E}_T\left(
\int_{0}^{\cdot }(\delta f^{\alpha_{s-}}(V_s^{v}))^{tr}dW_{s}\right)
\mathcal{E}_T\left(\int_0^{\cdot}\sum_{k,k^{\prime}\in I}\delta
g^{k'k}(V_s^{v})\chi_{\{\alpha_{s-}=k^{\prime}\}}d\tilde{N}_s^{k^{\prime}k}\right),
\end{equation*}%
so that under $\mathbb{Q}$, we have
\begin{equation}\label{equality_lambda}
\delta\lambda=\frac{\mathbb{E}^{\mathbb{Q}}
\left[\delta\mathcal{Y}_T^{v}-\delta\mathcal{Y}_0^{v}\right]}{T}
=\frac{\mathbb{E}^{\mathbb{Q}}\left[\mathbf{y}^{\alpha_T}(V_T^v)-\bar{\mathbf{y}}^{\alpha_T}(V_T^v)\right]-\left[\mathbf{y}^i(v)-\bar{\mathbf{y}}^i(v)\right]}{T}.
\end{equation}
Since both $\mathbf{y}^{k}(\cdot)$ and
$\bar{\mathbf{y}}^{k}(\cdot)$, $k\in I$, have at most linear growth
(cf. (\ref{estmate_y_1})), it follows from (\ref{first_estimate})
that $\delta\lambda=0$ by sending $T\rightarrow\infty$ in
(\ref{equality_lambda}).


We are left to show that
$\mathbf{y}^i(\cdot)=\bar{\mathbf{y}}^i(\cdot)$ and
$\mathbf{z}^i(\cdot)=\bar{\mathbf{z}}^i(\cdot)$ for $i\in I$. To
this end, it suffices to show that
\begin{equation}\label{uniqueness_Y}
\delta\mathcal{Y}^{v}_0={\mathcal{Y}}^v_0-\bar{\mathcal{Y}}^v_0=(\mathbf{y}^{i}-\bar{\mathbf{y}}^{i})(v)=0.
\end{equation}
The rest of the proof then follows from Theorem 3.11 in \cite{HU1}.
To prove (\ref{uniqueness_Y}), we have, from
(\ref{difference_ergodic_BSDE_jumps}), that
$$
\delta\mathcal{Y}_0^{v}=\mathbb{E}^{\mathbb{Q}}[\delta\mathcal{Y}_T^v]=
\mathbb{E}^{\mathbb{Q}}[\mathbf{y}^{\alpha_T}(V_T^{v})-\bar{\mathbf{y}}^{\alpha_T}(V_T^{v})].$$
Using (\ref{second_estimate}) and the fact that
$\mathbf{y}^i(0)=\bar{\mathbf{y}}^i(0)=0$, we obtain
$$\mathbb{E}^{\mathbb{Q}}[\mathbf{y}^{\alpha_T}(V_T^{v})-\bar{\mathbf{y}}^{\alpha_T}(V_T^{v})]
\leq C(1+|v|^2)e^{-K_vT}.$$ Hence, (\ref{uniqueness_Y}) follows by
sending $T\rightarrow \infty$ in the above inequality.

To conclude the paper, we recall the following moment estimate and
coupling estimate, which can be proved in a similar way to
\cite{HU1} (Proposition 2.3 and Theorem 2.4 for the Brownian motion
case), \cite{Cohen1} (section 3 for the Markov chain case) and
\cite{Cohen2} (section 3.2 for the L\'evy process case).

\begin{proposition}\label{coupling_lemma}
Let $T>0$ be fixed. Let $H^i:\mathbb{R}^d\rightarrow\mathbb{R}^d$
and $G^{ik}: \mathbb{R}^d\rightarrow\mathbb{R}$, $i,k\in I$, be
measurable bounded functions. Under Assumption \ref{assumption_13},
suppose that the processes $(V^v,\alpha)$ follow
$$dV_t^v=[\eta(V_t^v)+H^{\alpha_{t-}}(V_t^v)]dt+\kappa dW_t^{\mathbb{Q}},$$
and $$d\alpha_t=\sum_{k\in
I}q^{\alpha_{t-}k}(k-\alpha_{t-})(1+G^{\alpha_{t-}k}(V_t^v))dt+\sum_{k,k'\in
I}(k-k')\chi_{\{\alpha_{t-}=k'\}}d\tilde{N}_t^{\mathbb{Q},k'k},$$
where $\mathbb{Q}$ is an equivalent probability measure defined as
\begin{equation*}
\frac{d\mathbb{Q}}{d{\mathbb{P}}}:=\mathcal{E}_T\left(
\int_{0}^{\cdot }(H^{\alpha_{s-}}(V_s^{v}))^{tr}dW_{s}\right)
\mathcal{E}_T\left(\int_0^{\cdot}\sum_{k,k^{\prime}\in
I}G^{k'k}(V_s^{v})\chi_{\{\alpha_{s-}=k^{\prime}\}}d\tilde{N}_s^{k^{\prime}k}\right),
\end{equation*}%
with $W^{\mathbb{Q}}:=W-\int_0^{\cdot}H^{\alpha_{t-}}(V_t^{v})dt$
and $\tilde{N}^{\mathbb{Q},k'k}:=\tilde{N}^{k'k}-\int_0^{\cdot}
q^{k'k}G^{k'k}(V_t^v)dt$, $k',k\in I$, being the corresponding
Brownian motion and compensated Poisson martingales under
$\mathbb{Q}$, respectively. Then, there exists a constant $C>0$ such
that for any measurable functions $\phi^{i}:
\mathbb{R}^d\rightarrow\mathbb{R}$ $(i\in I)$ with a polynomial growth
rate $\mu>0$,
\begin{equation}\label{first_estimate}
\mathbb{E}^{\mathbb{Q}}[\phi^{\alpha_T}(V_T^v)]\leq C(1+|v|^{\mu}), \quad v\in\mathbb{R}^d.
\end{equation}
Furthermore, there exists a constant $K_{v}>0$ such that for $v_1,v_2\in \mathbb{R}^d,$
\begin{equation}\label{second_estimate}
\mathbb{E}^{\mathbb{Q}}[\phi^{\alpha_T}(V_T^{v_1})-\phi^{\alpha_T}(V_T^{v_2})]\leq
C(1+|v_1|^{1+\mu}+|v_2|^{1+\mu})e^{-K_{v}T}.
\end{equation}
The constants $C$ and $K_{v}$ depend on the functions $H^{i}(\cdot)$
and $G^{ik}(\cdot), i,k\in I,$ through their supremum norms.
\end{proposition}

\bigskip

{\bf Acknowledgment.} The authors thank the Editor, the Associate Editor and two referees for their helpful comments and suggestions. The authors also thank T. Zariphopoulou for stimulating discussions about forward performance processes,
	which motivate the current project.


\end{document}